\documentclass[12pt,draft]{amsart}
\author{Kevin Ford}
\author{Paul Pollack}
\address{Department of Mathematics\\University of Illinois\\1409 West Green Street\\Urbana, Illinois 61801}
\thanks{The first author was supported by NSF Grant DMS-0901339. The second author was supported by an
NSF Postdoctoral Fellowship (award DMS-0802970). The research was conducted in part while the authors
were visiting the Institute for Advanced Study, the first author supported by grants from the Ellentuck Fund and The Friends of the Institute For Advanced Study. Both authors thank the IAS for its hospitality and
excellent working conditions.}
\title{On common values of $\phi(n)$ and $\sigma(m)$, II}
\numberwithin{equation}{section}
\usepackage{amsmath,amssymb,amsthm}
\usepackage[margin=1in]{geometry}
\newtheorem{thm}{Theorem}[section]
\newtheorem{prop}[thm]{Proposition}

\newtheorem*{thmA}{Theorem A}

\newtheorem{lem}[thm]{Lemma}
\theoremstyle{remark}

\newtheorem*{rmks}{Remarks}

\DeclareMathAlphabet{\curly}{U}{rsfs}{m}{n}

\newcommand{\RR}{{\mathbb R}}

\renewcommand{\a}{\ensuremath{\alpha}}

\newcommand{\fancyA}{{\curly{A}}}

\newcommand{\fancyS}{{\curly S}}

\newcommand{\fancyR}{{\curly{R}}}

\renewcommand{\le}{\leqslant}
\renewcommand{\leq}{\leqslant}
\renewcommand{\ge}{\geqslant}
\renewcommand{\geq}{\geqslant}

\newcommand{\vone}{{\mathbf{1}}}

\renewcommand{\d}{\delta}
\renewcommand{\rho}{\varrho}

\newcommand{\bxi}{\boldsymbol\xi}
\newcommand{\vxi}{\bxi}

\newcommand{\vx}{\mathbf{x}}

\newcommand{\vy}{\mathbf{y}}

\begin{document}
\renewcommand{\a}{\ensuremath{\alpha}}
\newcommand{\pfrac}[2]{\left(\frac{#1}{#2}\right)}

\newcommand{\be}{\begin{equation}}
\newcommand{\ee}{\end{equation}}
\newcommand{\benn}{\begin{equation*}}
\newcommand{\eenn}{\end{equation*}}
\newcommand{\bal}{\begin{align*}}
\newcommand{\ea}{\end{align*}}
\newcommand{\eal}{\ensuremath{\end{align*}}}
\newcommand{\bea}{\begin{eqnarray}}
\newcommand{\eea}{\end{eqnarray}}

\renewcommand{\labelenumi}{(\roman{enumi})}
\def\d{\delta}
\def\D{\curly{D}}
\def\A{\curly{A}}
\def\fancyA{\curly{A}}
\def\J{\curly{J}}
\def\V{\curly{V}}
\def\N{\mathbf{N}}
\def\Q{\mathbf{Q}}
\def\Z{\mathbf{Z}}
\def\Cc{\curly{C}}
\def\Pp{\curly{P}}
\def\Ss{\fancyS}
\newcommand{\om}{\Omega}

\begin{abstract} For each positive-integer valued arithmetic function $f$, let $\V_{f}
 \subset \N$ denote the image of $f$, and put $\V_{f}(x) := \V_f\cap [1,x]$
 and $V_f(x) := \#\V_f(x)$. Recently Ford, Luca, and Pomerance showed that $\V_{\phi}\cap\V_{\sigma}$ is infinite, where $\phi$ denotes Euler's totient function and $\sigma$ is the usual sum-of-divisors function. Work of Ford shows that $V_{\phi}(x) \asymp V_{\sigma}(x)$ as $x\to\infty$. Here we prove a result complementary to that of Ford et al., by showing that most $\phi$-values are not $\sigma$-values, and vice versa. More precisely, we prove that as $x\to\infty$,
\[ \#\{n\leq x: n \in \V_{\phi} \cap \V_{\sigma}\} \leq \frac{V_{\phi}(x)+V_{\sigma}(x)}{(\log\log{x})^{1/2+o(1)}}. \]
\end{abstract}


\maketitle

\section{Introduction}
\subsection{Summary of results} For each positive-integer valued arithmetic function $f$, let $\V_{f}
 \subset \N$ denote the image of $f$, and put $\V_{f}(x) := \V_f\cap [1,x]$
 and $V_f(x) := \#\V_f(x)$. In this paper we are primarily concerned with
 the cases when $f=\phi$, the Euler totient function, and when $f=\sigma$,
 the usual sum-of-divisors function. When $f=\phi$, the study of the
 counting function $V_f$ goes back to Pillai \cite{pillai29}, and was
 subsequently taken up by Erd\H{o}s \cite{erdos35, erdos45}, Erd\H{o}s and
 Hall \cite{EH73, EH76},  Pomerance \cite{pomerance86}, Maier and Pomerance
 \cite{MP88}, and Ford \cite{ford98} (with an announcement in
 \cite{ford98A}). From the sequence of results obtained by these authors,
we mention Erd\H{o}s's asymptotic formula (from \cite{erdos35}) for
 $\log\frac{V_f(x)}{x}$, namely
\begin{equation}\label{eq:erdosestimate0}
 V_f(x) = \frac{x}{(\log{x})^{1+o(1)}} \quad (x\to\infty) \end{equation}
and the much more intricate determination of the precise order of magnitude by Ford,
\begin{equation}\label{eq:fordestimate0} V_f(x) \asymp \frac{x}{\log{x}} \exp(C(\log_3{x}-\log_4{x})^2+D\log_3{x}-(D+1/2-2C)\log_4{x}). \end{equation}
Here $\log_k$ denotes the $k$th iterate of the natural logarithm, and the constants $C$ and $D$ are defined as follows: Let
\begin{equation}\label{eq:somedefs} F(z) :=\sum_{n=1}^{\infty} a_n z^n, \quad\text{where}\quad a_n = (n+1)\log(n+1)-n\log{n}-1.\end{equation}
Since each $a_n >0$ and $a_n \sim \log{n}$ as $n\to\infty$, it follows that $F(z)$ converges
to a continuous, strictly increasing function on $(0,1)$, and $F(z)\to\infty$ as $z\uparrow 1$. Thus, there is a unique real number $\rho$ for which
\begin{equation}\label{eq:rhodef} F(\rho) = 1 \quad (\rho = 0.542598586098471021959\ldots).\end{equation}
In addition, $F'$ is strictly increasing, and
$F'(\rho) = 5.697758 \ldots$.
Then $C = \frac{1}{2|\log \rho|} = 0.817814\ldots$ and
$D = 2C(1+\log F'(\rho) - \log(2C)) - 3/2 = 2.176968 \ldots$.
In \cite{ford98}, it is also shown that \eqref{eq:fordestimate0} holds for a wide class of $\phi$-like functions, including $f=\sigma$. Consequently,  $V_\phi(x) \asymp V_{\sigma}(x)$.

Erd\H{o}s (see \cite[8, p. 172]{erdos59} or \cite{EG80}) asked if it could be proved that infinitely many natural numbers appear in both $\V_{\phi}$ and $\V_{\sigma}$. This question was recently answered by Ford, Luca, and Pomerance \cite{FLP10}. Writing $V_{\phi,\sigma}(x)$ for the number of common elements of $\V_\phi$ and $\V_\sigma$ up to $x$, they proved that
\[ V_{\phi,\sigma}(x) \geq \exp((\log\log{x})^{c}) \]
for some positive constant $c > 0$ and all large $x$
(in \cite{Ga} this is shown for \emph{all} constants $c>0$).
This lower bound is probably very far from the truth. For example, if $p$ and $p+2$ form a twin prime pair, then $\phi(p+2)=p+1=\sigma(p)$; a quantitative form of the twin prime conjecture then implies that $V_{\phi,\sigma}(x) \gg x/(\log{x})^2$. In Part I, we showed that a stronger conjecture of the same type allows for an improvement. Roughly, our result is as follows:

\begin{thmA} Assume a strong uniform version of Dickson's prime $k$-tuples conjecture. Then as $x\to\infty$,
\[ V_{\phi,\sigma}(x) = \frac{x}{(\log{x})^{1+o(1)}}. \]
\end{thmA}

Theorem A suggests that $V_{\phi,\sigma}(x)$ is much larger than we might naively expect. This naturally leads one to inquire about what can be proved in the opposite direction; e.g., could it be that a positive proportion of $\phi$-values are also $\sigma$-values? The numerical data up to $10^{9}$, exhibited in Table \ref{tbl:data}, suggests that the proportion of common values is decreasing, but the observed rate of decrease is rather slow.
\begin{table}\label{tbl:data}
\begin{tabular}{l||r|r|r|r|r}
$N$ & $V_{\phi}(N)$ & $V_{\sigma}(N)$ & $V_{\phi,\sigma}(N)$ & $V_{\phi,\sigma}(N)/V_{\phi}(N)$ & $V_{\phi,\sigma}(N)/V_{\sigma}(N)$\\\hline
   10000   &    2374   &    2503   &   1368 &  0.5762426  & 0.5465441\\
   100000  &    20254  &    21399  &   11116 &  0.5488299  & 0.5194635\\
  1000000  &   180184  &   189511  &   95145  & 0.5280436  & 0.5020553\\
 10000000  &  1634372  &  1717659  &  841541  & 0.5149017  & 0.4899348\\
 100000000 &  15037909 &  15784779 &  7570480 &  0.5034264 &  0.4796063\\
1000000000 & 139847903 & 146622886 & 69091721 &  0.4940490 &  0.4712206\\
\end{tabular}
\caption{Data on $\phi$-values, $\sigma$-values, and common values up to $N=10^k$, from $k=5$ to $k=9$.}
\end{table}
Our principal result is the following estimate, which implies in particular that almost all $\phi$-values are not $\sigma$-values, and vice versa.
\begin{thm}\label{thm:main} As $x\to\infty$,
\begin{equation}\label{eq:main} V_{\phi,\sigma}(x)\leq \frac{V_{\phi}(x)+V_{\sigma}(x)}{(\log\log{x})^{1/2+o(1)}}. \end{equation}
\end{thm}
The proof of Theorem \ref{thm:main} relies on the detailed structure theory of totients as developed in \cite{ford98}. It would be interesting to know the true rate of decay of $V_{\phi,\sigma}(x)/V_{\phi}(x)$.

\subsection{Sketch}
Since the proof of Theorem \ref{thm:main} is rather intricate and involves a number of technical estimates, we present a brief outline of the argument in this section.

We start by discarding a sparse set of undesirable $\phi$ and $\sigma$-values. More precisely, we identify (in Lemma \ref{lem:capturemost}) convenient sets $\A_{\phi}$ and $\A_{\sigma}$ with the property that almost all $\phi$-values~$\leq x$ have all their preimages in $\A_{\phi}$ and almost all $\sigma$-values $\leq x$ have all their preimages in $\A_{\sigma}$. This reduces us to studying how many $\phi$ and $\sigma$-values arise as solutions to the equation
\[ \phi(a) = \sigma(a'), \quad\text{where}\quad a \in \A_{\phi}, a' \in \A_{\sigma}. \]
Note that to show that $V_{\phi,\sigma}(x)/V_{\phi}(x)\to 0$, we need only count the number of common $\phi$-$\sigma$-values of this kind, and not the (conceivably much larger) number of pairs $(a,a') \in \A_{\phi}\times \A_{\sigma}$ corresponding to these values.

What makes the sets $\A_{\phi}$ and $\A_{\sigma}$ convenient for us? The properties imposed in the definitions of these sets are of two types, \emph{anatomical} and \emph{structural}. By anatomical considerations, we mean general considerations of multiplicative structure as commonly appear in elementary number theory (e.g., consideration of the number and size of prime factors). By structural considerations, we mean those depending for their motivation on the fine structure theory of totients developed by Ford \cite{ford98}.

Central to our more anatomical considerations is the notion of a \emph{normal prime}. Hardy and Ramanujan \cite{HR00} showed that almost all natural numbers $\leq x$ have $\sim \log\log{x}$ prime factors, and Erd\H{o}s \cite{erdos35} showed that the same holds for almost all shifted primes $p-1 \leq x$. Moreover, sieve methods imply that if we list the prime factors of $p-1$ on a double-logarithmic scale, then these are typically close to uniformly distributed in $[0,\log\log{p}]$. Of course, all of this remains true with $p+1$ in place of $p-1$. We assume that the numbers belonging to $\A_{\phi}$ and $\A_{\sigma}$ have all their prime factors among this set of normal primes.

If we assume that numbers $n$ all of whose prime factors are normal generate ``most'' $f$-values (for $f\in\{\phi, \sigma\}$), we are led to a series of linear inequalities among the (double-logarithmically renormalized) prime factors of $n$. These inequalities are at the heart of the structure theory of totients as developed in \cite{ford98}. As one illustration of the power of this approach,
mapping the $L$ largest prime factors of $n$ (excluding the largest) to a point in $\RR^L$, the problem of
estimating $V_f(x)$ reduces to the problem of finding the volume of a certain region of $\RR^L$, called the fundamental simplex.
%% Essentially,
%$$
%V_f(x) \approx \frac{x}{\log x} \max_{L} T_L (\log_2 x)^L,
%$$
%where $T_L$ denotes the volume of the simplex.  It turns out that the maximum
%occurs at $L=L_0(x) + O(1)$, where here and below,
%\[ L_0(x) := \lfloor 2C(\log_3{x}-\log_4{x})\rfloor. \]
%(Here $C$ is as defined in the introduction.)
%
In broad strokes, this is how one establishes Ford's bound \eqref{eq:fordestimate0}. We incorporate these linear inequalities into our definitions of $\A_{\phi}$ and $\A_{\sigma}$. One particular linear combination of renormalized prime factors appearing in the definition of the fundamental simplex is of particular interest to us (see condition (8) in the definition of $\A_{f}$ in \S\ref{sec:afdef} below); that we can assume this quantity is $< 1$ is responsible for the success of our argument.

%We emphasize that these observations have many further applications, including the determination of the normal order of the number of prime factors of a typical totient (see \cite[Theorem 10]{ford98} for details).

Suppose now that we have a solution to $\phi(a) = \sigma(a')$, where $(a, a') \in \A_{\phi} \times \A_{\sigma}$. We write $a = p_0 p_1 p_2 \cdots$ and $a'= q_0 q_1 q_2 \cdots$, where the sequences of $p_i$ and $q_j$ are non-decreasing. We cut the first of these lists in two places; at the $k$th prime $p_k$ and at the $L$th place $p_L$. The precise choice of $k$ and $L$ is somewhat technical; one should think of the primes $p_i$ larger than $p_k$ as the ``large'' prime divisors of $a$, those smaller than $p_k$ but larger than $p_L$ as ``small', and those smaller than $p_L$ as ``tiny''. The equation $\phi(a) = \sigma(a')$ can be rewritten in the form
\begin{equation}\label{eq:formtocount} (p_0-1)(p_1-1)(p_2-1)\cdots (p_{k-1}-1)fd = (q_0+1)(q_1+1)(q_2+1)\cdots(q_{k-1}+1)e, \end{equation}
where
\begin{equation}\label{eq:fdedefs0} f:= \phi(p_k \cdots p_{L-1}), \quad d:=\phi(p_L p_{L+1}\cdots), \quad\text{and}\quad e:= \sigma(q_k q_{k+1} \cdots). \end{equation}
To see that \eqref{eq:formtocount} correctly expresses the relation $\phi(a)=\sigma(a')$, we recall that the primes $p_1, \dots, p_k$ are all large, so that by the ``anatomical'' constraints imposed in the definition of $\A_{\phi}$, each appears to the first power in the prime factorization of $a$. An analogous statement holds for the primes $q_1, \dots, q_k$; this follows from the general principle, established below, that $p_i \approx q_i$  provided that either side is not too small. 
There is one respect in which \eqref{eq:fdedefs0} may not be quite right: Since $p_L$ is tiny, we cannot assume a priori that $p_L \neq p_{L-1}$, and so it may be necessary to amend the definition of $d$ somewhat; we ignore this (ultimately minor) difficulty for now.

To complete the argument, we fix $d$ and estimate from above the number of solutions (consisting of $p_0, \dots, p_{k-1}, q_0, \dots, q_{k-1}, e, f$) to the relevant equations of the form \eqref{eq:formtocount}; then we sum over $d$. The machinery facilitating these estimates is encoded in Lemma \ref{lem51}, which is proved by a delicate, iterative sieve argument of a kind first introduced by Maier and Pomerance \cite{MP88} and developed in \cite[\S5]{ford98}. The hypotheses of that lemma include several assumptions about the $p_i$ and $q_j$, and about $e$, $f$, and $d$. All of these rather technical hypotheses are, in our situation, consequences of our definitions of $\A_{\phi}$ and $\A_{\sigma}$; we say more about some of them in a remark following Lemma \ref{lem51}.

\subsection*{Notation} Let $P^+(n)$ denote the largest prime factor of $n$, understood so that $P^{+}(1)=1$, and let $\om(n,U,T)$ denote the total number of prime factors $p$ of $n$
such that $U < p \le T$, counted according to multiplicity.
Constants implied by the Landau $O-$ and Vinogradov $\ll-$ and $\gg-$
symbols are absolute unless otherwise specified.
Symbols in boldface type indicate vector quantities.

\section{Preliminaries}\label{sec:prelim}

\subsection{Anatomical tools}
We begin with two tools from the standard chest. The first is a form of the upper bound sieve and the second concerns the distribution of smooth numbers.

\begin{lem}[see, e.g., {\cite[Theorem 4.2]{HR74}}]\label{sieve linear factors}
 Suppose $a_1,\ldots,a_h$ are positive integers and
$b_1,\ldots,b_h$ are integers such that
$$
E = \prod_{i=1}^h a_i \prod_{1\le i<j\le h} (a_ib_j-a_jb_i) \ne 0.
$$
Then
$$
\#\{ n\le x: a_in+b_i \text{ prime } (1\le i\le h) \}
\ll \frac{x}{(\log x)^h} \prod_{p|E}
\frac{1-\frac{\nu(p)}{p}}{(1-\frac{1}{p})^h}
\ll  \frac{x(\log_2 (|E|+2))^h}{(\log x)^{h}},
$$
where $\nu(p)$ is the number of solutions of the congruence
$\prod (an_i+b_i)\equiv 0\pmod{p}$, and
the implied constant may depend on $h$.
\end{lem}

Let $\Psi(x,y)$ denote the number of $n\leq x$ for which $P^{+}(n) \leq y$. The following estimate is due to Canfield, Erd\H{o}s, and Pomerance \cite{CEP83}:

\begin{lem}\label{lem:CEP} If $2\leq y \leq x$ and $u = \log{x}/\log{y}$, then
\[ \Psi(x,y) = x/u^{u+o(u)} \]
for $u \leq y^{1-\epsilon}$, as $u\to\infty$.
\end{lem}

The next lemma supplies an estimate for how often $\Omega(n)$ is unusually large; this may be deduced from the Theorems in Chapter 0 of \cite{HT88}.

\begin{lem} \label{Omega lem}
The number of integers $n \le x$ for which $\om(n) \ge \a \log_2 x$ is
$$
\ll_\a \begin{cases} x (\log x)^{-Q(\a)} & \text{if $1<\a<2$}, \\
x (\log x)^{1-\a \log 2}\log_2 x & \text{if $\a \ge 2$}, \end{cases}
$$
where $Q(\lambda) = \int_1^{\lambda} \log{t}\, dt = \lambda \log(\lambda)-\lambda+1$.
\end{lem}

We also require a simple estimate for the decay of the Poisson distribution near its left and right tails.

\begin{lem}[see {\cite[Lemma 2.1]{ford98}}] \label{exp partial} If $z>0$ and $0<\alpha<1$, then
$$
\sum_{k\le \alpha z} \frac{z^k}{k!} < \left(\frac{e}{\alpha}\right)^{\alpha z} = e^{(1-Q(\alpha))z},
$$
where $Q(\cdot)$ is defined as in Lemma \ref{Omega lem}.
\end{lem}
\begin{proof} We have
$$
\sum_{k\le \alpha z} \frac{z^k}{k!} =
\sum_{k\le \alpha z} \frac{(\alpha z)^k}{k!} \pfrac{1}{\alpha}^k \le
\pfrac{1}{\alpha}^{\alpha z} \sum_{k \le \alpha z} \frac{(\alpha z)^k}{k!}
< \pfrac{e}{\alpha}^{\alpha z} = e^{(1-Q(\alpha))z}. \qedhere
$$
 \end{proof}

In the remainder of this section, we give a precise meaning to the term ``normal prime'' alluded to in the introduction and draw out some simple consequences. For $S\ge 2$, a prime $p$ is said to be \emph{$S$-normal} if the following two conditions holds for each $f \in \{\phi, \sigma\}$:
\be\label{1S}
\om(f(p),1,S) \le 2\log_2 S,
\ee
and, for every pair of real numbers $(U,T)$ with $S \le U<T\le f(p)$, we have
\be\label{normal}
|\om(f(p),U,T) - (\log_2 T - \log_
2 U)| < \sqrt{\log_2 S \log_2 T}.
\ee
This definition is slightly weaker than the corresponding definition on \cite[p. 13]{ford98}, and so the results from that paper remain valid in our context. As a straightforward consequence of the definition, if $p$ is $S$-normal, $f \in \{\phi, \sigma\}$, and $f(p) \geq S$, then
\begin{equation}\label{eq:normalpomega} \Omega(f(p)) \leq 3 \log_2{f(p)}. \end{equation}

The following lemma is a simple consequence of \cite[Lemma 2.10]{ford98} and \eqref{eq:fordestimate0}:

\begin{lem}\label{LEM210} For each $f \in \{\phi, \sigma\}$, the number of $f$-values $\leq x$ which have a preimage divisible by a prime that is not $S$-normal is
\[ \ll V_f(x) (\log_2{x})^{5} (\log{S})^{-1/6}. \]
\end{lem}

We also record the observation that if $p$ is $S$-normal, then $P^{+}(f(p))$ cannot be too much smaller than $p$, on a double-logarithmic scale.

\begin{lem}\label{lem:length2} If $5 \leq p \leq x$ is an $S$-normal prime and $f(p) \geq S$, then
\[ \frac{\log_2{P^{+}(f(p))}}{\log_2{x}} \geq \frac{\log_2{p}}{\log_2{x}} - \frac{\log_3{x}+\log{4}}{\log_2{x}}. \]
\end{lem}
\begin{proof} We have $P^{+}(f(p)) \geq f(p)^{1/\Omega(f(p))} \geq f(p)^{\frac{1}{3\log_2{f(p)}}} \geq p^{\frac{1}{4\log_2{x}}}$. The result follows upon taking the double logarithm of both sides.
\end{proof}

\subsection{Structural tools}\label{sec:structure} In this section, we describe more fully some components of the structure theory of totients alluded to in the introduction. Given a natural number $n$, write $n = p_0(n) p_1(n) p_2(n) \cdots$, where $p_0(n) \geq p_1(n) \geq p_2(n) \dots$ are the primes dividing $n$ (with multiplicity). For a fixed $x$, we put
\[ x_i(n; x) =
 \begin{cases} \frac{\log_2{p_i(n)}}{\log_2{x}} &\text{if $i < \Omega(n)$ and $p_i(n) > 2$}, \\
 0 &\text{if $i \geq \Omega(n)$ or $p_i(n)=2$}.
\end{cases}
\]
Suppose $L \geq 2$ is fixed and that $\xi_i\ge 0$ for $0\le i\le L-1$. Recall the definition of the $a_i$ from \eqref{eq:somedefs} and
let $\fancyS_L(\bxi)$ be the set of $(x_1, \ldots, x_L) \in\RR^L$ with $0 \leq x_L \leq x_{L-1} \leq \dots \leq x_1 \leq 1$ and
\begin{align*}
(I_0) & &a_1x_1 + a_2x_2 + \cdots + a_L x_L &\le \xi_0, \hfill \\
(I_1) & &a_1x_2 + a_2x_3 + \cdots + a_{L-1}x_L &\le \xi_1 x_1, \\
\vdots\;\;\; && &\vdots\\
(I_{L-2}) & &a_1 x_{L-1} + a_2 x_{L} &\le \xi_{L-2} x_{L-2}.
%(I_{L-1}) & & 0\le x_{L} &\le \xi_{L-1} x_{L-1}.
\end{align*}
%and let $\fancyS_L(\bxi)$ be the subset of $\fancyS_L^*(\bxi)$ satisfying
%$0 \le x_L \le \cdots \le x_1 \le 1$.
Define $T_L(\bxi)$ as the volume ($L$-dimensional Lebesgue measure) of $\fancyS_L(\bxi)$. For convenience, let $\vone=(1,1,\ldots,1)$,
$\fancyS_L = \fancyS_L(\vone)$ (the
``fundamental simplex''), and let $T_L$  be the volume of $\fancyS_L$. Let
\[ L_0(x):= \lfloor 2C(\log_3{x}-\log_4{x})\rfloor, \]
where $C$ is defined as in the introduction. The next lemma, which appears as \cite[Theorem 15]{ford98}, allows us to locate the preimages of almost all $f$-values within suitable sets of the form $\fancyS_L(\vxi)$.
\begin{lem}\label{lem:simplex}
Suppose $0 \leq \Psi < L_0(x)$, $L= L_0-\Psi$, and let
\[ \xi_i = \xi_i(x) = 1 + \frac{1}{10(L_0-i)^3} \qquad (0 \leq i \leq L-2). \]
The number of $f$-values $v \leq x$ with a preimage $n$ for which
$(x_1(n; x), \dots, x_L(n;x)) \not\in \fancyS_L(\vxi)$
is
\[ \ll V_f(x) \exp(-\Psi^2/4C). \]
\end{lem}

For future use, we collect here some further structural lemmas from \cite{ford98}. The next result, which follows immediately from \eqref{eq:fordestimate0} and \cite[Lemma 4.2]{ford98}, concerns the size of sums of the shape appearing in the definition of inequality ($I_0$) above.

\begin{lem}\label{lem:lem42} Suppose $L \geq 2$, $0 < \omega < 1/10$, and $x$ is sufficiently large. The number of $f$-values $v \leq x$ with a preimage satisfying
\[ a_1 x_1(n; x) + \dots + a_L x_L(n; x) \geq 1 + \omega \]
is \[ \ll V_f(x) (\log_2{x})^5 (\log{x})^{-\omega^2/(150 L^3 \log{L})}.\]
\end{lem}

We will make heavy use of the following (purely geometric) statement about the simplices $\fancyS_L(\vxi)$; it appears as \cite[Lemma 3.10]{ford98}.

\begin{lem}\label{lem:xcomparison} If $\vx \in \fancyS_L(\bxi)$ and $\xi_0^{L}\xi_1^{L-1}\cdots \xi_{L-2}^2 \leq 1.1$, then $x_j \leq 3\rho^{j-i} x_i$ when $i < j$, and $x_j < 3\rho^j$ for $1 \leq j \leq L$.
\end{lem}

Define $\fancyR_L(\vxi; x)$ as the set of integers $n$ with $\Omega(n) \leq L$ and
\[ (x_0(n;x), x_1(n;x), \dots, x_{L-1}(n;x)) \in \fancyS_L(\vxi). \]
For $f \in \{\phi,\sigma\}$, put
\[ R_L^{(f)}(\vxi; x) = \sum_{n \in \fancyR_L(\vxi, x)} \frac{1}{f(n)}. \]
The next lemma, extracted from \cite[Lemma 3.12]{ford98}, relates the magnitude of $R_L^{(f)}(\vxi; x)$ to the volume of the fundamental simplex $T_L$, whenever $\bxi$ is suitably close to $\vone$. In \cite{ford98}, it plays a crucial role in the proof of the upper-bound aspect of \eqref{eq:fordestimate0}.

\begin{lem}\label{lem:RLbound1} If $1/(1000k^3) \leq \omega_{L_0-k} \leq 1/(10k^3)$ for $1 \leq k \leq L_0$, $\xi_i = 1 + \omega_i$ for each $i$, and $L \leq L_0$, then
\[ R_L^{(f)}(\vxi; x) \ll (\log_2{x})^L T_L \]
for both $f=\phi$ and $f=\sigma$.
\end{lem}

While only the case $f=\phi$ of Lemma \ref{lem:RLbound1} appears in the statement of \cite[Lemma 3.12]{ford98}, the $f=\sigma$ case follows trivially, since $\sigma(n)\geq \phi(n)$. In order to apply Lemma \ref{lem:RLbound1}, we need estimates for the volume $T_L$; this is handled by the next lemma, extracted from \cite[Corollary 3.4]{ford98}.

\begin{lem}\label{lem:volbound} Assume $1 \leq \xi_i \leq 1.1$ for $0 \leq i \leq L-2$ and that $\xi_0^L \xi_1^{L-1} \cdots \xi_{L-2}^{2} \asymp 1$. If $L= L_0 - \Psi > 0$, then
\[ (\log_2{x})^{L} T_L(\vxi) \ll Y(x) \exp(-\Psi^2/4C). \]
Here
\begin{equation}\label{eq:ydef} Y(x):= \exp(C(\log_3{x}-\log_4{x})^2 + D\log_3{x} -(D+1/2-2C)\log_4{x}).\end{equation}
\end{lem}

We conclude this section with the following technical lemma, which will be needed when we select the sets $\A_{\phi}$ and $\A_{\sigma}$ in \S\ref{sec:afdef}.
\begin{lem}\label{lem:smoothness} For $f \in \{\phi, \sigma\}$ and $y\geq 20$,
\begin{equation}\label{eq:totientrecip} \sum_{\substack{v \in \V_f \\ P^{+}(v) \leq y}}\frac{1}{v} \ll \frac{\log_2{y}}{\log_3{y}} Y(y),\end{equation}
where $Y$ is as defined in \eqref{eq:ydef}. Moreover, for any $b > 0$,
\begin{equation}\label{eq:Yexpbounds}
 Y(\exp((\log{x})^{b})) \ll_{b} Y(x) \left(\frac{\log_3{x}}{\log_2{x}}\right)^{-2C\log{b}}.\end{equation}
\end{lem}
\begin{proof} We split the left-hand sum in \eqref{eq:totientrecip} according to whether or
 not $v \leq y^{\log_2{y}}$. The contribution of the large $v$ is $O(1)$ and so is negligible: 
Indeed, for $t > y^{\log_2{y}}$, we have $\frac{\log{t}}{\log{y}} > \log_2{y}$. 
Thus, by Lemma \ref{lem:CEP}, we have $\Psi(t,y) \ll t/(\log{t})^2$ (say), and the 
$O(1)$ bound follows by partial summation. We estimate the sum over small $v$ by ignoring the smoothness condition. Put $X=y^{\log_2{y}}$. Since $V_f(t) \asymp \frac{t}{\log{t}} Y(t)$, partial summation gives that
\[ \sum_{\substack{v \in \V_f \\ v \leq X }}\frac{1}{v} \ll 1 + 
\int_{3}^{X}\frac{Y(t)}{t\log{t}}\, dt = (1+o(1)) Y(X) \frac{\log_2{X}}{\log_3{X}}, \]
as $y \to \infty$. (The last equality follows, e.g., from L'H\^{o}pital's rule.) Since $\log_2{X}/\log_3{X} \sim \log_2{y}/\log_3{y}$ and $Y(X) \sim Y(y)$, we have \eqref{eq:totientrecip}. Estimate \eqref{eq:Yexpbounds} follows from the definition of $Y$ and a direct computation; here it is helpful to note that if we redefine $X:=\exp((\log{x})^b)$, then $\log_3{X} = \log_3{x} + \log{b}$ and $\log_4{X} = \log_4{x} + O_b(1/\log_3{x})$.
\end{proof}

\section{Definition of the sets $\A_{\phi}$ and $\A_{\sigma}$}\label{sec:afdef}
We continue fleshing out the introductory sketch, giving precise definitions to the preimage sets $\A_{\phi}$ and $\A_{\sigma}$.
Put
\begin{equation}\label{eq:Ldef} L:=\lfloor L_0(x)-2\sqrt{\log_3{x}}\rfloor, \quad \xi_i:= 1 + \frac{1}{10(L_0-i)^3} \quad (1\leq i \leq L). \end{equation}

The next lemma is a final technical preliminary.

\begin{lem}\label{lem:lotsofprimes} Let $f \in \{\phi, \sigma\}$. The number of $f$-values $v \leq x$ with a preimage $n$ for which
\begin{enumerate}
\item $(x_1(n;x), \dots, x_L(n;x)) \in \fancyS_L(\vxi)$,
\item $n$ has fewer than $L+1$ odd prime divisors (counted with multiplicity),
\end{enumerate}
is $\ll V_{f}(x)/\log_2{x}$.
\end{lem}
\begin{proof} We treat the case when $f=\phi$; the case when $f=\sigma$ requires only small modifications.
We can assume that $x/\log{x} \leq n \leq 2x\log_2{x}$, where the last inequality follows from known results on the minimal order of the Euler function. By Lemma \ref{Omega lem}, we can also assume that $\Omega(n) \leq 10\log_2{x}$. Put $p_i := p_i(n)$, as defined in \S\ref{sec:structure}. Since $(x_1(n;x), \dots, x_L(n;x)) \in \fancyS_L(\vxi)$ by hypothesis, Lemma \ref{lem:xcomparison} gives that $x_2 < 3 \rho^2 < 0.9$, and so $p_2 \leq \exp((\log{x})^{0.9})$. Thus,
\[ n/(p_0 p_1) = p_2 p_3 \cdots \leq \exp(10 (\log_2{x}) (\log{x})^{0.9}) = x^{o(1)},  \]
and so $p_0 \geq x^{2/5}$ (say) for large $x$. In particular, we can assume that $p_0^2 \nmid n$.

Suppose now that $n$ has exactly $L_0-k+1$ odd prime factors, where we fix $k > L_0-L$. Then
\[ v = (p_0-1)\phi(p_1 p_2 \cdots p_{L_0-k}) 2^s \]
for some integer $s\geq 0$. Using the prime number theorem to estimate the number of choices for $p_0$ given $p_1 \cdots p_{L_0-k}$ and $2^s$, we obtain that the number of $v$ of this form is
\[ \ll \frac{x}{\log{x}} \sum_{p_1 \cdots p_{L_0-k}}\frac{1}{\phi(p_1 \cdots p_{L_0-k})} \sum_{s \geq 0} \frac{1}{2^s}. \]
(We use here that $\frac{x}{\phi(p_1 \cdots p_{L_0-k}) 2^s} \gg p_0 \geq x^{2/5}$.) The sum over $s$ is $\ll 1$. To handle the remaining sum, we observe that $p_1 \cdots p_{L_0-k}$ belongs to the set $\curly R_{L_0-k}(\vxi_k, x)$, where $\vxi_k:= (\xi_0, \dots, \xi_{L_0-k-2})$. Thus, the remaining sum is bounded by
\[ R_{L_0-k}^{(\phi)}(\vxi_k; x) = \sum_{m \in \curly R_{L_0-k}(\vxi_k, x)}\frac{1}{\phi(m)}.  \]
So by Lemmas \ref{lem:RLbound1} and \ref{lem:volbound}, both of whose hypotheses are straightforward to verify,
\[  R_{L_0-k}^{(\phi)}(\vxi_k; x) \ll (\log_2{x})^{L_0-k} T_{L_0-k} \leq (\log_2{x})^{L_0-k}
 T_{L_0-k}(\vxi_k) \ll Y(x) \exp(-k^2/4C).
\]
Collecting our estimates, we obtain a bound of
\[ \ll \frac{x}{\log{x}} Y(x) \exp(-k^2/4C) \ll V_{\phi}(x) \exp(-k^2/4C).  \]

Now since $L_0 - L > 2\sqrt{\log_3{x}}$, summing over $k > L_0-L$ gives a final bound which is
\[ \ll V_{\phi}(x) \exp(-(\log_3{x})/C) \ll V_{\phi}(x)/\log_2{x}, \]
as desired.
\end{proof}

For the rest of this paper, we fix $\epsilon > 0$ and assume that $x \geq x_0(\epsilon)$. Put
\begin{equation}\label{eq:Sdeltaomegadef} S:= \exp((\log_2{x})^{36}), \quad \delta := \sqrt{\frac{\log_2{S}}{\log_2{x}}}, \quad \omega := (\log_2{x})^{-1/2+\epsilon/2}. \end{equation}
For $f \in \{\phi, \sigma\}$, let $\A_f$ be the set of $n=p_0(n) p_1(n) \cdots$ satisfying $f(n) \leq x$ and
\begin{enumerate}
\item[(0)] $n \geq x/\log{x}$,
\item[(1)] every squarefull divisor $m$ of $n$ or $f(n)$ satisfies $m \leq \log^2{x}$,
\item[(2)] all of the primes $p_j(n)$ are $S$-normal,
\item[(3)] $\Omega(f(n)) \leq 10\log_2{x}$ and $\Omega(n) \leq 10\log_2{x}$,
\item[(4)] If $d \parallel n$ and $d \geq \exp((\log_2{x})^{1/2})$, then $\Omega(f(d)) \leq 10 \log_2{f(d)}$,
\item[(5)] $(x_1(n;x), \dots, x_L(n;x)) \in \Ss_L(\vxi)$,
\item[(6)] $n$ has at least $L+1$ odd prime divisors,
\item[(7)] $P^{+}(f(p_0)) \geq x^{\frac{1}{\log_2{x}}}$, $p_1(n) < x^{\frac{1}{100\log_2{x}}}$,
\item[(8)] $a_1 x_1 + \dots + a_L x_L \leq 1-\omega$.
\end{enumerate}

The following lemma asserts that a generic $f$-value has all of its preimage in $\A_f$.

\begin{lem}\label{lem:capturemost} For each $f \in \{\phi, \sigma\}$, the number of $f$-values $\leq x$ with a preimage $n\not\in \A_f$ is
\[ \ll V_f(x) (\log_2{x})^{-1/2+\epsilon}.\]
\end{lem}

\begin{rmks}\mbox{}
\begin{enumerate}
\item The $\A_f$ not only satisfy Lemma \ref{lem:capturemost} but do so economically. In fact, from condition (5) and the work of \cite[\S4]{ford98}, we have that $\#\A_f \ll V_f(x)$. Thus, on average, an element of $\V_f(x)$ has only a bounded number of preimages from $\A_f$.  So when we turn in \S\S\ref{sec:fundsieve}-- \ref{sec:application} to counting $\phi$-values arising from solutions to $\phi(a) = \sigma(a')$, with $(a, a') \in \A_{\phi}\times\A_{\sigma}$, we expect not to be (excessively) overcounting.
\item Of the nine conditions defining $\A_f$, conditions (0)--(4) are, in the nomenclature of the introduction, purely \emph{anatomical}, while conditions (5)-(8) depend to some degree on the fine structure theory of \cite{ford98}. Condition (8) is particularly critical. It is (8) which ensures that the sieve bounds developed in \S\ref{sec:fundsieve} result in a nontrivial estimate for $V_{\phi,\sigma}(x)$. Our inability to replace the exponent $\frac{1}{2}$ on $\log_2{x}$ in Lemma \ref{lem:capturemost} (or in Theorem \ref{thm:main}) by a larger number is also rooted in (8).
\end{enumerate}
\end{rmks}

\begin{proof} It is clear that the number of values of $f(n)$ corresponding to $n$ failing (0) or (1) is $\ll x\log_2{x}/\log{x}$, which (recalling \eqref{eq:fordestimate0}) is permissible for us. By Lemma \ref{LEM210} and our choice of $S$, the number of values of $f(n)$ coming from $n$ failing (2) is $\ll V_f(x)/\log_2{x}$. The same holds for values coming from $n$ failing (3), by Lemma \ref{Omega lem}.

Suppose now that $n$ fails condition (4). Then $n$ has a unitary divisor $d \geq \exp((\log_2{x})^{1/2})$ with $\Omega(f(d)) \geq 10\log_2{f(d)}$. Put $w:=f(d)$. Then $w \mid f(n)$, and $f(n) \ll x\log_2{x}$. So if $w \geq x^{1/2}$, then the number of possibilities for $f(n)$ is
\[ \ll x\log_2{x} \sum_{\substack{w \geq x^{1/2} \\ \Omega(w) \geq 10\log_2{w}}}\frac{1}{w} \ll \frac{x\log_2{x}}{\log{x}}, \]
using Lemma \ref{Omega lem} to estimate sum over $w$. If $w \leq x^{1/2}$, we observe that $f(n)/w = f(n/d) \in \V_f$ and deduce that the number of corresponding values of $f(n)$ is
\[ \ll \sum_{\substack{\exp((\log_2{x})^{1/3}) \leq w \leq x^{1/2} \\ \Omega(w) \geq 10\log_2{w}}}V_f(x/w) \ll \frac{x}{\log{x}}Y(x) \sum_{\substack{w \geq \exp((\log_2{x})^{1/3}) \\ \Omega(w) \geq 10\log_2{w}}}\frac{1}{w} \ll \frac{V_f(x)}{\log_2{x}}.\]

By Lemma \ref{lem:simplex}, the number of  $f$-values with a preimage failing (5) is $\ll V_f(x)/\log_2{x}$.
 According to Lemma \ref{lem:lotsofprimes}, the number of $f$-values with a preimage satisfying (5) but not (6) is also $\ll V_f(x)/\log_2{x}$.

Suppose now that $n$ satisfies (0)--(6). From (5), we have $1 + \xi_0 \geq a_1x_1 + a_2 x_2 \geq (a_1 + a_2) x_2$, and so $x_2 \leq 0.8$. So from (3),
\begin{equation}\label{eq:p0large}
 \frac{n}{p_0(n)p_1(n)} = p_2(n) p_3(n) \cdots \leq \exp(10(\log_2{x})(\log{x})^{0.8}) < x^{1/100}.
\end{equation}
In particular, $p_0 > x^{1/3}+1$ and $f(p_0) > x^{1/3}$, so that $v:=f(p_1 p_2 \cdots) \leq x^{2/3}$. The prime $p_0$ satisfies $f(p_0) \leq x/v$. For $z$ with $x^{1/3} < z \leq x$, the number of primes $p_0$ with $f(p_0) \leq z$ and $P^{+}(f(p_0)) \leq x^{1/\log_2{x}}$ is (crudely) bounded by $\Psi(z,x^{1/\log_2{x}}) \ll z/(\log{x})^2$, by Lemma \ref{lem:CEP}. So the number of values of $f(n)$ coming from $n$ with $P^{+}(f(p_0)) \leq x^{1/\log_2{x}}$ is
\[ \ll \sum_{\substack{v \leq x^{2/3}\\ v\in \V_f}} \sum_{\substack{p:~f(p) \leq x/v\\ P^{+}(f(p)) \leq x^{1/\log_2{x}}}} 1 \ll \frac{x}{(\log{x})^2} \sum_{\substack{v\leq x^{2/3} \\ v \in \V_f}}\frac{1}{v} \ll \frac{x}{(\log{x})^{2-\epsilon}}. \]
To handle the second condition in (7), observe that
since $f(p_0) \leq x/v$, the prime number theorem (and the bound $v \leq x^{2/3}$) shows that given $v$,
the number of possibilities for $p_0$ is $\ll x/(v\log{x})$. Suppose that $p_1(n) > x^{\frac{1}{100\log_2{x}}}$.
%We can assume $P^{+}(f(p)) \geq x^{1/\log_2{x}}$, by what was shown above.
Then $x_1 = x_1(n; x) \geq 0.999$, and we conclude from $\sum_{i \geq 1}{a_i x_i}\leq \xi_0$ that either $x_2 \leq \rho^{3/2}$ or $x_3 \leq \rho^{5/2}$. Writing $v_2$ for $f(p_2 p_3 \cdots)$ and $v_3$ for $f(p_3 p_4 \cdots)$, we see that the number of such $f$-values is
\begin{multline*} \ll \frac{x}{\log{x}}\sum_{p_1} \frac{1}{p_1}\left(\sum_{P^{+}(v_2) \leq \exp((\log{x})^{\rho^{3/2}})} \frac{1}{v_2} + \sum_{p_2}\frac{1}{p_2}\sum_{P^{+}(v_3) \leq \exp((\log{x})^{\rho^{5/2}})}\frac{1}{v_3}\right) \\
\ll \frac{x}{\log{x}} \log_3{x} \left(Y(x) \left(\frac{\log_3{x}}{\log_2{x}}\right)^{1/2} + (\log_2{x})Y(x)\left(\frac{\log_3{x}}{\log_2{x}}\right)^{3/2} \right)  \ll \frac{V_f(x)}{(\log_2{x})^{1/2-\epsilon}},
\end{multline*}
using Lemma \ref{lem:smoothness} to estimate the sums over $v_2$ and $v_3$.

Finally, we consider $n$ for which (0)--(7) hold but where condition (8) fails. By Lemma \ref{lem:lem42}, we can assume that
\[ a_1 x_1 + \cdots + a_L x_L < 1+\omega, \]
since the number of exceptional $f$-values is $\ll V(x)\exp(-(\log_2{x})^{\epsilon/2})\ll V(x)/\log_2{x}$. Thus,
\begin{equation}\label{eq:squeezed} 1-\omega < a_1 x_1 + \cdots + a_L x_L < 1+\omega, \end{equation}
while by condition ($I_1$) in the definition of $\fancyS_L(\vxi)$,
\[ a_1 x_2 + \cdots + a_{L-1} x_{L} \leq \xi_1 x_1. \]
We claim that if $J$ is fixed large enough depending on $\epsilon$, then there is some $2 \leq j \leq J$ with $x_j \leq \rho^{j-\epsilon/3}$. If not, then for large enough $J$,
\[ \xi_1 x_1 \geq \sum_{j=1}^{J-1} a_{i} x_{i+1} \geq \rho^{1-\epsilon/3} (a_1 \rho + a_2 \rho^2 + \cdots + a_J \rho^{J}) > \rho^{1-\epsilon/4}. \]
Thus, $x_1 \geq \rho^{1-\epsilon/4} \xi_1^{-1} \geq \rho^{1-\epsilon/5}$, and so
\[ \xi_0 \geq \rho^{-\epsilon/5} (a_1 \rho + a_2\rho^2 + \dots + a_J \rho^{J}) \geq \rho^{-\epsilon/6},\]
which is false. This proves the claim. We assume below that $j \in [2,J]$ is chosen as the smallest index with $x_j \leq \rho^{j-\epsilon/3}$; by condition (1), this implies that all of $p_1, \dots, p_{j-1}$ appear to the first power in the prime factorization of $n$.

Now given $x_2, \dots, x_L$, we have from \eqref{eq:squeezed} that $x_1 \in [\alpha, \alpha+2\omega]$ for a certain $\alpha$. Thus,
\[ \sum_{p_1}\frac{1}{p_1} \ll \omega\log_2{x} = (\log_2{x})^{1/2+\epsilon/2}. \]
So the number of $f$-values that arise from $n$ satisfying (0)--(7) but failing (8) is
\begin{multline*}
\ll \frac{x}{\log{x}} \sum_{j=2}^{J}\sum_{p_1, \dots, p_{j-1}} \frac{1}{p_1 \cdots p_{j-1}} \sum_{\substack{P^{+}(v) \leq \exp((\log{x})^{\rho^{j-\epsilon/3}})\\ v \in \V_f}}\frac{1}{v} \\
\ll \frac{x}{\log{x}} \sum_{j=2}^{J}(\log_2{x})^{j-3/2 + \epsilon/2} Y(x) \left(\frac{\log_3 x}{\log_2 x}\right)^{-1+j-\epsilon/3} \ll V_f(x)(\log_2{x})^{-1/2+\epsilon}.
\end{multline*}
This completes the proof of Lemma \ref{lem:capturemost}.
\end{proof}

As a corollary of Lemma \ref{lem:capturemost}, we have that $V_{\phi,\sigma}(x)$ is bounded, up to an additive error of $\ll (V_\phi(x)+V_{\sigma}(x))/(\log_2{x})^{1/2-\epsilon}$, by the number of values $\phi(a)$ that appear in solutions to the equation
\[ \phi(a) = \sigma(a'), \quad\text{where}\quad (a,a') \in \A_{\phi}\times\A_{\sigma}. \]
In \S\S\ref{sec:fundsieve}--\ref{sec:application}, we develop the machinery required to estimate the number of such values. Ultimately, we find that it is smaller than $(V_{\phi}(x)+V_{\sigma}(x))/(\log_2{x})^A$ for any fixed $A$, which immediately gives Theorem \ref{thm:main}.

\section{The fundamental sieve estimate}\label{sec:fundsieve}
\begin{lem}\label{lem51}  Let $y$ be large, $k\geq 1$, $l\geq 0$, $e^e \leq S \leq v_k \leq v_{k-1} \leq \dots \leq v_0 = y$, and $u_j \leq v_j$ for $0 \leq j \leq k-1$. Put $\delta =\sqrt{\log_2{S}/\log_2{y}}$, $\nu_j = \log_2{v_j}/\log_2{y}$, $\mu_j = \log_2{u_j}/\log_2{y}$. Suppose that 
%$1 \leq d \leq y^{1/100}$
$d$ is a natural number for which $P^{+}(d) \leq v_k$. Moreover, suppose that both of the following hold:
\begin{enumerate}
%\item[(a)] $v_1 \leq y^{1/10\log_2 y}$,
\item[(a)] For $2 \leq j \leq k-1$, either $(\mu_j,\nu_j) = (\mu_{j-1},\nu_{j-1})$ or $\nu_j \leq \mu_{j-1} -2\delta$. Also, $\nu_{k}\leq \mu_{k-1} - 2\delta$.
\item[(b)] For $1 \leq j \leq k-2$, we have $\nu_j > \nu_{j+2}$.
\end{enumerate}
The number of solutions of
\begin{equation}\label{eq:piqi} (p_0-1)\cdots (p_{k-1}-1) f d = (q_0+1)\cdots (q_{k-1}+1) e \leq y \end{equation}
in $p_0, \dots, p_{k-1}, q_0, \dots, q_{k-1}, e, f$ satisfying
\begin{enumerate}
\item $p_i$ and $q_i$ are $S$-normal primes;
\item $u_i \leq P^{+}(p_i-1), P^{+}(q_i+1) \leq v_i$ for $0 \leq i \leq k-1$;
\item neither $\phi(\prod_{i=0}^{k-1}{p_i})$ nor $\sigma(\prod_{i=0}^{k-1}{q_i})$ is divisible by $r^2$ for a prime $r \geq v_k$;
\item $P^{+}(ef) \leq v_k$; $\Omega(f) \leq 4l\log_2{v_k}$;
\item $p_0-1$ has a divisor $\geq y^{1/2}$ which is composed of primes $> v_1$;
\end{enumerate}
is
\[ \ll \frac{y}{d} (c\log_2{y})^{6k} (k+1)^{\Omega(d)} (\log{v_k})^{8(k+l)\log{(k+1)}+1} (\log{y})^{-2+\sum_{i=1}^{k-1}a_i \nu_i +E}, \]
where $E = \delta \sum_{i=2}^{k}{(i\log{i}+i)} + 2 \sum_{i=1}^{k-1} (\nu_i-\mu_i)$.
Here $c$ is an absolute positive constant.
\end{lem}
\begin{rmks} Since the lemma statement is very complicated, it may be helpful to elaborate on how it will be applied in \S\ref{sec:application} below.
Given $(a,a') \in \A_{\phi}\times\A_{\sigma}$ satisfying $\phi(a) = \sigma(a')$, rewrite the corresponding equation in  the form \eqref{eq:formtocount}, with $d$, $e$, and $f$  as in \eqref{eq:fdedefs0}. (Here $L$ is as in \eqref{eq:Ldef}, and $k$, given more precisely
in the next section, satisfies $k \approx L/2$.) We are concerned with counting the number of values $\phi(a)$ which arise from such solutions. We partition the solutions according to the value of $d$, which describes
the contribution of the ``tiny'' primes to $\phi(a)$, and by the rough location of the primes $p_i$ and $q_i$, which we
encode in the selection of intervals $[u_i, v_i]$ (cf. Lemma \ref{lem:length2}). Finally, we apply Lemma \ref{lem51} and sum over both $d$ and
the possible selections of intervals; this gives an estimate for the number of $\phi(a)$ which is smaller than $(V_{\phi}(x) + V_{\sigma}(x))/(\log_2{x})^A$,
for any fixed $A$.

In our application, conditions (i)--(v) of Lemma \ref{lem51} are either immediate from the definitions, or are readily deduced from the defining properties of
$\A_{\phi}$ and $\A_{\sigma}$. Conditions (a) and (b) are rooted in the observation that while neighboring primes in the prime factorization of $a$ (or $a'$) may be
close together (requiring us to allow $[u_{i+1}, v_{i+1}] = [u_{i}, v_{i}]$), the primes $p_i(a)$ and $p_{i+2}(a)$ are forced
to be far apart on a double-logarithmic scale. Indeed, since $(x_1(a; x), \dots, x_L(a;x)) \in \fancyS_L(\vxi)$, Lemma \ref{lem:xcomparison} shows that
$x_{i+2} < 3\rho^2 x_i < 0.9 x_i$.
\end{rmks}

\begin{proof}
We consider separately the prime factors of each shifted prime lying in
each interval $(v_{i+1},v_{i}]$.
For  $0\le j \le k-1$ and $0\le i\le k$, let
\[
s_{i,j}(n) = \prod_{\substack{ p^a\parallel (p_j-1) \\p\le v_i }} p^a, \qquad
s'_{i,j}(n) = \prod_{\substack{ p^a\parallel (q_j+1) \\p\le v_i }} p^a, \qquad
s_i = d f \prod_{j=0}^{k-1} s_{i,j} = e\prod_{j=0}^{k-1} s'_{i,j}.
\]
Also, for $0\le j\le k-1$ and $1\le i\le k$, let
\[
t_{i,j} = \frac{s_{i-1,j}}{s_{i,j}}, \qquad t'_{i,j} = \frac{s'_{i-1,j}}{s'_{i,j}},
\qquad t_i = \prod_{j=0}^{k-1} t_{i,j} = \prod_{j=0}^{k-1} t'_{i,j}.
\]
For each solution $\fancyA=(p_0,\ldots,p_{k-1},f,q_0,\ldots,q_{k-1},e)$
of \eqref{eq:piqi}, let
\begin{align*}
\sigma_i(\fancyA) &= \{s_i;s_{i,0},\ldots,
s_{i,k-1},f; s_{i,0}', \ldots,
s_{i,k-1}',e\},\\
\tau_i(\fancyA) &= \{t_i; t_{i,0},\ldots,t_{i,k-1},1;
t'_{i,0},\ldots,t'_{i,k-1},1\}.
\end{align*}
Defining multiplication of $(2k+l+2)$-tuples by component-wise multiplication,
we have
\be\label{sigtau}
\sigma_{i-1}(\fancyA) = \sigma_i(\fancyA) \tau_i(\fancyA).
\ee
Let $\mathfrak S_i$ denote the set of $\sigma_i(\fancyA)$ arising from solutions
$\fancyA$ of \eqref{eq:piqi} and $\mathfrak T_i$ the corresponding set of
$\tau_i(\fancyA)$.
By \eqref{sigtau}, the number of solutions of \eqref{eq:piqi} satisfying the required
conditions is
\be\label{S0}
|\mathfrak S_0| =
\sum_{\sigma_1 \in \mathfrak S_1} \sum_{\substack{ \tau_1 \in \mathfrak T_1 \\
\sigma_1\tau_1 \in \mathfrak S_0 }} 1.
\ee

First, fix $\sigma_1\in \mathfrak S_1$.  By assumption (v) in the lemma,
$t_{1,0} \ge y^{1/2}$.  Also, $t_1=t_{1,0}=t_{1,0}' \le y/s_1$, $t_1$
is composed of primes $> v_1$, and also $s_{1,0}t_1+1$ and $s'_{1,0}t_1-1$
are prime.
Write $t_1=t_1'Q$, where $Q=P^+(t_1)$.
Since $p_0$ is an $S$-normal prime, \eqref{eq:normalpomega} gives that
\[ Q\ge t_1^{1/\Omega(t_1)} \ge t_1^{1/\Omega(p_0-1)} \geq y^{1/(2\Omega(p_0-1))} \geq
y^{1/(6\log_2 y)},\]
Given $t_1'$, Lemma \ref{sieve linear factors} implies that the number of
$Q$ is $O(y(\log_2 y)^6/(s_1 t_1' \log^3 y))$.  Moreover,
\[ \sum\frac{1}{t_1'} \leq \prod_{v_1< p\leq y} \left(1+\frac{1}{p}+\frac{1}{p^2}+\dots\right) \ll \frac{\log{y}}{\log{v_1}} = (\log{y})^{1-\nu_1}.\]
Consequently, for each $\sigma_1\in \mathfrak S_1$,
\be\label{sigma1}
\sum_{\substack{ \tau_1 \in \mathfrak T_1 \\ \sigma_1\tau_1 \in \mathfrak S_0 }}
 1 \ll \frac{y (\log_2 y)^6}{s_1 (\log y)^{2+\nu_1}}.
\ee

Next, suppose $2\le i\le k$.  We now apply an iterative procedure: If $v_{i} < v_{i-1}$, we use the identity
\be\label{st iter}
\sum_{\sigma_{i-1} \in \mathfrak S_{i-1}} \frac1{s_{i-1}} =
\sum_{\sigma_i \in \mathfrak S_i} \frac1{s_i}
 \sum_{\substack{ \tau_i \in \mathfrak T_i \\ \sigma_i\tau_i \in
\mathfrak S_{i-1} } }  \frac1{t_i}.
\ee
If $v_{i} = v_{i-1}$, then \eqref{st iter} remains true but contains no information, and in this case we use the alternative identity
\be\label{st iter2}
\sum_{\sigma_{i-1} \in \mathfrak S_{i-1}} \frac1{s_{i-1}} =
\sum_{\sigma_{i+1} \in \mathfrak S_{i+1}} \frac1{s_{i+1}}
 \sum_{\substack{ \tau_{i+1} \in \mathfrak T_{i+1} \\ \sigma_{i+1}\tau_{i+1} \in
\mathfrak S_{i-1} } }  \frac1{t_{i+1}}.
\ee

We consider first the simpler case when $v_i < v_{i-1}$. Suppose $\sigma_i\in \mathfrak S_i$,
$\tau_i\in \mathfrak T_i$ and $\sigma_i\tau_i\in \mathfrak S_{i-1}$.
By assumption (ii),
$$
t_i=t_{i,0} \cdots t_{i,i-1} = t'_{i,0} \cdots t'_{i,i-1}.
$$
In addition, $s_{i,i-1}t_{i,i-1}+1=p_{i-1}$ and
 $s'_{i,i-1}t'_{i,i-1}-1=q_{i-1}$
are prime. Let $Q := P^+(t_{i,i-1})$, $Q' := P^+(t'_{i,i-1})$,
$b:=t_{i,i-1}/Q$ and $b':=t'_{i,i-1}/Q'$.

We consider separately $\mathfrak T_{i,1}$, the set of $\tau_i$ with $Q=Q'$
and $\mathfrak T_{i,2}$, the set of $\tau_i$ with $Q \ne Q'$.
First,
$$
\Sigma_1 := \sum_{\substack{ \tau_i\in \mathfrak T_{i,1} \\
\sigma_i\tau_i \in \mathfrak S_{i-1}
} } \frac1{t_i} \le \sum_t \frac{h(t)}{t}\max_{b,b'} \sum_{Q} \frac1{Q},
$$
where $h(t)$ denotes the number of solutions of
$t_{i,0} \cdots t_{i,i-2} b = t = t'_{i,0} \cdots t'_{i,i-2} b'$,
and in the sum on $Q$,
$s_{i,i-1}bQ+1$ and $s'_{i,i-1}b'Q-1$ are prime.
By Lemma \ref{sieve linear factors}, the
number of $Q\le z$ is $\ll z(\log z)^{-3} (\log_2 y)^3$ uniformly in
$b,b'$.  By partial summation,
\[
\sum_{Q \ge u_{i-1}} \frac{1}{Q} \ll (\log_2 y)^3 (\log y)^{-2\mu_{i-1}}.
\]
Also, $h(t)$ is at most the number of dual factorizations of $t$ into
$i$ factors each, i.e., $h(t) \le i^{2\Omega(t)}$.
By \eqref{normal},
$\Omega(t) \le i(\nu_{i-1} - \nu_{i} +\d)\log_2 y =: I.$
Also, by assumption (iii), $t$ is squarefree.
Thus,
\begin{equation}\label{eq:ht0}
\sum_t \frac{h(t)}{t} \le \sum_{j\le I} \frac{i^{2j}H^j}{j!},\end{equation}
where
$$
\sum_{v_{i} < p \le v_{i-1}} \frac{1}{p} \le (\nu_{i-1} -
 \nu_{i})\log_2 y + 1 =: H.
$$
By assumption (a), $\nu_{i-1}-\nu_{i} \ge 2\d,$
hence $I \le \frac32 iH \le \frac34 i^2H$.
Applying Lemma \ref{exp partial}  (with $\a \leq \frac34$) to estimate the right-hand side of \eqref{eq:ht0}, we find that \be\label{sum h(t)}
\sum_t \frac{h(t)}{t} \le \pfrac{eHi^2}{I}^I \le (ei)^I
= (\log y)^{(i+i\log i)(\nu_{i-1}-\nu_{i} + \d)}.
\ee
This gives
\[
\Sigma_1 \ll (\log_2 y)^3 (\log y)^{-2\mu_{i-1} +
(i+i\log i)(\nu_{i-1}-\nu_i + \d) }.
\]

For the sum over $\mathfrak T_{i,2}$, set $t_{i}=tQQ'$.
Note that
$$
t Q' = t_{i,0} \cdots t_{i,i-2}b, \qquad tQ = t'_{i,0} \cdots
t'_{i,i-2} b',
$$
so $Q\mid t'_{i,0} \cdots t'_{i,i-2} b'$ and $Q' \mid t_{i,0} \cdots t_{i,i-2}b$.
If we fix the factors divisible by $Q$ and by $Q'$, then the
number of possible ways to form $t$ is $\le i^{2\Omega(t)}$ as before.
Then
$$
\Sigma_2 := \sum_{\substack{ \tau_i\in \mathfrak T_{i,2} \\ \sigma_i\tau_i\in \mathfrak S_{i-1}}} \frac{1}{t_i}
 \le \sum_t \frac{i^{2\Omega(t)+2}}{t} \max_{b,b'} \sum_{Q,Q'} \frac1{Q Q'},
$$
where $s_{i,i-1}bQ+1$ and $s'_{i,i-1}b'Q'-1$ are prime.
By Lemma \ref{sieve linear factors}, the
number of $Q\le z$ (respectively $Q'\le z$) is $\ll z(\log z)^{-2}
(\log_2 y)^2$. Hence,
$$
\sum_{Q,Q'} \frac1{Q Q'} \ll (\log_2 y)^4 (\log y)^{-2\mu_{i-1}}.
$$
Combined with \eqref{sum h(t)}, this gives
$$
\Sigma_2 \ll i^2 (\log_2 y)^4 (\log y)^{-2 \mu_{i-1} +
(i+i\log i)(\nu_{i-1}-\nu_{i}+\d) }.
$$
From (a) and (b), $i^2 \le k^2 \le (\log_2 y)^2$.  Adding
$\Sigma_1$ and $\Sigma_2$ shows that for each $\sigma_i$,
\be\label{siti}
\sum_{\substack{ \tau_i \in \mathfrak T_i \\  \sigma_i\tau_i \in \mathfrak S_{i-1} }} \frac1{t_i}
\ll (\log_2 y)^6 (\log y)^{-2\mu_{i-1} +
 (i\log i + i)(\nu_{i-1}-\nu_{i}+\d)}.
\ee

We consider now the case when $v_{i} = v_{i-1}$. Set $Q_1 := P^{+}(t_{i+1,i-1})$, $Q_2 := P^{+}(t_{i+1,i})$, $Q_3 := P^{+}(t_{i+1,i-1}')$, and $Q_4 := P^{+}(t_{i+1,i}')$. From (iii), we have that $Q_1 \neq Q_2$ and $Q_3 \neq Q_4$. Moreover, letting $b_i$ denote the cofactor of $Q_i$ in each case, we have that
\begin{alignat}{2}
s_{i+1,i-1} b_1 Q_1 + 1 &= p_{i-1}, &\qquad\qquad s_{i+1,i-1}' b_3 Q_3 - 1 &= q_{i-1},\notag\\
s_{i+1,i} b_2 Q_2 + 1 &= p_{i}, &\qquad\qquad s_{i+1,i}' b_4 Q_4 - 1 &= q_{i}.\label{eq:psandqs}
\end{alignat}
Since there are now several ways in which the various $Q_i$ may coincide, the combinatorics is more complicated than in the case when $v_i < v_{i-1}$. We index the cases by fixing the incidence matrix $(\delta_{ij})$ with $\delta_{ij}=1$ if $Q_i=Q_j$ and $\delta_{ij}=0$ otherwise.

Write $D = \gcd(Q_1 Q_2, Q_3 Q_4)$, and let $Q := Q_1 Q_2/D$ and $Q':=Q_3 Q_4/D$, so that $D, Q$, and $Q'$ are formally determined by $(\delta_{ij})$. Then $QQ' \mid t_{i+1}$, and writing $t_{i+1}/D = tQQ'$, we have
\begin{align}\label{eq:tQ} t Q&= t_{i+1,0} t_{i+1,1} \cdots t_{i+1,i-2}b_3 b_4, \\
\label{eq:tQprime} tQ'&=  t_{i+1,0}' t_{i+1,1}' \cdots t_{i+1,i-2}' b_1 b_2. \end{align}
We now choose which terms on the right-hand sides of \eqref{eq:tQ} and \eqref{eq:tQprime} contain the prime factors of $Q$ and $Q'$, respectively; since $\Omega(Q)\leq 2$ and $\Omega(Q') \leq 2$, this can be done in at most $(i+1)^4$ ways. Having made this choice, the number of ways to form $t$ is bounded by $(i+1)^{2\Omega(t)}$, and so
\begin{align} \sum_{\substack{\tau_{i+1} \in \frak T_{i+1} \\ \sigma_{i+1} \tau_{i+1} \in \frak S_{i-1}}} \frac{1}{t_{i+1}} &\leq \sum_{t} \frac{(i+1)^{2\Omega(t)+4}}{t} \max_{b_1, b_2, b_3, b_4} \sum \frac{1}{D Q Q'}\label{eq:tequals}.\end{align}
It is easy to check that $DQQ'=\prod_{j \in \J}{Q_j}$, where $\J$ indexes the distinct $Q_j$. For each $j \in \J$, let $n_j$ be the number of linear forms appearing in \eqref{eq:psandqs} involving $Q_j$. Since each of these $n_j$ linear forms in $Q_j$ is prime, as is $Q_j$ itself, Lemma \ref{sieve linear factors} implies that the number of possibilities for $Q_j \leq z$ is $\ll z(\log{z})^{-n_j-1} (\log_2{y})^{n_{j}+1}$, and so
\[ \sum_{Q_j \geq u_{i-1}} \frac{1}{Q_j} \ll (\log_2{y})^{n_j+1} (\log{u_{i-1}})^{-n_j} \ll (\log_2{y})^{n_j+1} (\log{y})^{-n_j \mu_{i-1}},\]
uniformly in the choice of the $b$'s. Since $\sum_{j \in \J}{n_j} =4$ and $\sum_{j \in \J}1 \leq 4$,
\begin{equation}\label{eq:dqqp} \sum\frac{1}{DQQ'} \leq \prod_{j \in \J} \left(\sum_{Q_j \geq u_{i-1}}\frac{1}{Q_j}\right) \ll (\log_2{y})^{8} (\log{y})^{-4\mu_{i-1}}. \end{equation} The calculation \eqref{sum h(t)}, with $i$ replaced by $i+1$, shows that
\begin{equation}\label{eq:hsum2} \sum_{t} \frac{(i+1)^{2\Omega(t)}}{t} \leq (\log{y})^{((i+1)+(i+1)\log(i+1))(\nu_i-\nu_{i+1}+\delta)}. \end{equation}
Combining \eqref{eq:tequals}, \eqref{eq:dqqp}, and \eqref{eq:hsum2} shows that
\begin{multline}
\sum_{\substack{\tau_{i+1} \in \frak T_{i+1} \\ \sigma_{i+1} \tau_{i+1} \in \frak S_{i-1}}} \frac{1}{t_{i+1}} \leq (i+1)^4 (\log_2{y})^8 (\log{y})^{-4\mu_{i-1}+ ((i+1)+(i+1)\log(i+1))(\nu_i-\nu_{i+1}+\delta)}
\\ \leq (\log_2{y})^{12} (\log{y})^{-2\mu_{i-1} + (i\log{i}+i)(\nu_i-\nu_{i-1})-2\mu_{i} + ((i+1)\log(i+1)+(i+1))(\nu_{i+1}-\nu_{i}+\delta)},\label{eq:case2}
\end{multline}
where in the last line we use that $v_{i-1} = v_i$ and $(i+1)^4 \leq k^4 \leq (\log_2{y})^4$.

Using \eqref{S0}, \eqref{st iter}, and \eqref{st iter2} together with the inequalities
\eqref{sigma1}, \eqref{siti}, and \eqref{eq:case2}, we find that the
the number of solutions of \eqref{eq:piqi} is
\[
\ll y (c \log_2 y)^{6k}
(\log y)^{-2-\nu_{1} + \sum_{i=2}^k (\nu_{i-1} - \nu_{i}+\d)(i\log i + i)
 - 2\mu_{i-1} } \sum_{\sigma_k \in \mathfrak S_k}
\frac1{s_k},
\]
where $c$ is some positive constant.  Note that the exponent
of $\log y$ is $\le -2+\sum_{i=1}^{k-1} a_i \nu_i + E$.

It remains to treat the sum on $\sigma_k$.
Given $s_k'=s_k/d$, the number of possible $\sigma_k$ is
 at most the number of factorizations of $s_k'$ into $k+1$ factors
times the number of factorizations of $ds_k'$ into $k+1$ factors,
which is at most $(k+1)^{\Omega(ds_k')} (k+1)^{\Omega(s_k')}$.
By assumptions (i) and (iv), $\Omega(s_k')\le 4(k+l)\log_2 v_k$.
Thus,
\[
\sum_{\sigma_k \in \mathfrak S_k} \frac1{s_k} \le \frac{(k+1)^{\Omega(d)}
(k+1)^{8(k+l)\log_2 v_k}}{d} \sum_{P^+(s_k')\le v_k} \frac{1}{s_k'}
\ll \frac{(k+1)^{\Omega(d)} (\log v_k)^{8(k+l)\log(k+1)+1}}{d}.
\qedhere
\]
\end{proof}

\section{Counting common values: Application of Lemma \ref{lem51}}\label{sec:application} In this section we prove the following proposition, which combined with Lemma \ref{lem:capturemost} immediately yields  Theorem \ref{thm:main}. Throughout the rest of this paper, we adopt the definitions of $L$, the $\xi_i$, $S$, $\delta$, and $\omega$ from  \eqref{eq:Ldef} and \eqref{eq:Sdeltaomegadef}.

\begin{prop}\label{prop:main} Fix $A > 0$. For large $x$, the number of distinct values of $\phi(a)$ that arise from solutions to the equation
\[ \phi(a) = \sigma(a'), \quad\text{with}\quad (a,a') \in \A_{\phi}\times\A_{\sigma}, \]
is smaller than $(V_{\phi}(x)+V_{\sigma}(x))/(\log_2{x})^{A}$.
\end{prop}

Let us once again recall the strategy outlined in the introduction and in the remarks following Lemma \ref{lem51}. Let $(a,a') \in \A_{\phi}\times\A_{\sigma}$
be a solution to $\phi(a) = \sigma(a')$. Let $p_i := p_i(a)$ and $q_i := p_i(a')$, in the notation of \S\ref{sec:structure}. We choose a cutoff  $k$
so that all of $p_0, \dots, p_{k-1}$ and $q_0, \dots, q_{k-1}$ are ``large''. Then by condition (1) in the definition of the sets $\A_f$, neither $p_i^2 \mid a$  nor $q_i^2 \mid a'$, for $0 \leq i \leq k-1$.  Fixing a notion of ``small'' and ``tiny'', we rewrite the equation $\phi(a) = \sigma(a')$ in the form
\begin{equation}\label{eq:formtocount2} (p_0-1)\cdots (p_{k-1}-1) fd = (q_0+1)\cdots(q_{k-1}+1)e, \end{equation}
where $f$ is the contribution to $\phi(a)$ from the ``small'' primes, $d$ is the contribution from the ``tiny'' primes, and $e$ is the contribution of both the ``small'' and ``tiny'' primes to $\sigma(a')$.

We then fix $d$ and numbers $u_i$ and $v_i$, chosen so that $u_i \leq P^{+}(p_i-1), P^{+}(q_i+1) \leq v_i$ for each $0 \leq i \leq k-1$.
With these fixed, Lemma \ref{lem51} provides us with an upper bound on the number of corresponding solutions to \eqref{eq:formtocount2}. Such a solution determines the common value
$\phi(a) = \sigma(a') \in \V_{\phi}\cap\V_{\sigma}$. We complete the proof of Proposition \ref{prop:main} by summing the upper bound estimates over all choices of $d$ and all
selections of the  $u_i$ and $v_i$.

We carry out this plan in four stages, each of which is treated in more detail below:
\begin{itemize}
\item Finalize the notions of ``small'' and ``tiny'', and so also the choices of $d$, $e$, and $f$.
\item Describe how to choose the $u_i$ and $v_i$ so that the intervals $[u_i, v_i]$ capture $P^{+}(p_i-1)$ and $P^{+}(q_i+1)$ for all $0 \leq i \leq k-1$.
\item Check that the hypotheses of Lemma \ref{lem51} are satisfied.
\item Take the estimate of Lemma \ref{lem51} and sum over $d$ and the choices of $u_i$ and $v_i$.\end{itemize}

 \subsection{``Small'' and ``tiny''}\label{sec:smalltiny} Suppose we are given a solution $(a,a') \in \A_{\phi}\times\A_{\sigma}$ to $\phi(a)=\sigma(a')$. Set $x_j = x_j(a; x)$ and $y_j = x_j(a'; x)$, in the notation of \S\ref{sec:structure}, so that (from the definition of $\A_f$) the sequences $\vx =(x_1, \dots, x_L)$ and $\vy = (y_1, \dots, y_L)$ belong to $\fancyS_L(\vxi)$.

\begin{lem}\label{lem:xbounds} With $\{z_j\}_{j=1}^{L}$ denoting either of the sequences $\{x_j\}$ or $\{y_j\}$, we have
\begin{enumerate}
\item $z_j < 3\rho^j$ for $1 \leq j \leq L$,
\item $z_{L-j} \geq \frac{3}{100} \rho^{-j}/\log_2{x}$ for $0 \leq j < L$.
\item $z_{j+2} \leq 0.9 z_j$ for $1 \leq j \leq L-2$.
\end{enumerate}
\end{lem}
\begin{proof} Claim (i) is repeated verbatim from Lemma \ref{lem:xcomparison}. By the same lemma, $z_j \leq 3\rho^{j-i} z_i$ for $1 \leq i < j \leq L$. This immediately implies (iii), since $\rho^2 < 0.9$. Moreover, fixing $j=L$, condition (6) in the definition of $\A_f$ gives that
\[ z_{i} \geq \frac{1}{3} \rho^{i-L} z_{L} \geq \frac{\log_2{3}}{3} \rho^{i-L}/\log_2{x} > \frac{3}{100} \rho^{i-L}/\log_2{x},\] which is (ii) up to a change of variables.
\end{proof}

\begin{lem}\label{lem:minimalindex} The minimal index $k_0 \leq L$ for which
\[ \log_2 P^{+}(p_{k_0}-1) < (\log_2{x})^{1/2+\epsilon/10}. \]
satisfies $k_0 \sim (1/2-\epsilon/10) L$ as $x\to\infty$.
\end{lem}
\begin{proof} Lemma \ref{lem:xbounds}(i) shows that the least $K$ with $\log_2 p_{K} < (\log_2{x})^{1/2+\epsilon/10}$ satisfies $K \leq (1/2-\epsilon/10+o(1))L$,
as $x\to\infty$. Since $\log_2 P^{+}(p_{K}-1) \leq \log_2 p_K$, this gives the asserted upper bound on $k_0$. The lower bound follows in a similar fashion from Lemma \ref{lem:xbounds}(ii) and Lemma \ref{lem:length2}.
\end{proof}

Recall the definition of $\delta$ from \eqref{eq:Sdeltaomegadef}, and put
\begin{equation}\label{eq:etadef}
 \eta:= 10L\delta,\quad\text{so that}\quad \eta \asymp (\log_3{x})^{3/2} (\log_2{x})^{-1/2}. \end{equation}
We choose our ``large''/``small'' cutoff point $k$ by taking $k=k_0$ if $x_{k_0-1} -x_{k_0} \geq 20\eta$, and taking $k=k_0-1$ otherwise.
For future use, we note that with this choice of $k$,
\begin{equation}\label{eq:finalseparation} x_{k-1} - x_k \geq 20 \eta. \end{equation}
This inequality is immediate if $k=k_0$; in the opposite case, by Lemma \ref{lem:xbounds}(iii),
\begin{align*} x_{k-1}-x_{k} = x_{k_0-2} - x_{k_0-1} &\geq x_{k_0-2}-x_{k_0}-20\eta \\&\geq 0.1 x_{k_0-2} -20\eta
\geq 0.1(\log_2{x})^{-1/2+\epsilon/10} - 20\eta > 20\eta.\end{align*}

Note that with this choice of $k$, we have $\log_2{p_i} > (\log_2{x})^{1/2+\epsilon/10}$ for $0 \leq i \leq k-1$,
and so condition (1) in the definition of $\A_{\phi}$ guarantees that each $p_i$ divides $a$ to the first power only, for $0 \leq i \leq k-1$.
Moreover, from Lemmas \ref{lem:xbounds}(ii) and \ref{lem:minimalindex}, we have $\log_2{q_i} > (\log_2{x})^{1/2+\epsilon/11}$ for $0 \leq i \leq k-1$.
So each $q_i$ divides $a'$ only to the first power, for $0 \leq i \leq k-1$. Now take
\begin{equation}\label{eq:fddefs} f := \phi(p_k p_{k+1} \cdots p_{L-1}), \quad d:=\begin{cases}
\phi(p_L p_{L+1} \cdots) &\text{if $p_{L-1} \neq p_{L}$}, \\
\frac{p_L}{\phi(p_L)}\phi(p_L p_{L+1} \cdots) &\text{if $p_{L-1} =p_{L}$},
\end{cases}\end{equation}
and
\[ e := \sigma(q_k q_{k+1} \cdots), \]
and observe that equation \eqref{eq:formtocount2} holds.
%(If $p_{L-1}^2 \mid a$, then this choice of $d$ differs from that of \eqref{eq:fdedefs0}.)

\subsection{Selection of the $u_j$ and $v_j$}\label{sec:ujvjchoice} Rather than choose the $u_j$ and $v_j$ directly, it is more convenient to work with the $\mu_j$ and $\nu_j$;
then $u_j$ and $v_j$  are defined by $u_j: = \exp((\log{x})^{\mu_j})$ and $v_j:=\exp((\log{x})^{\nu_j})$. Put
\begin{equation}\label{eq:zetajdef} \zeta_0 := 1-\frac{\log_3{x} + \log{100}}{\log_2{x}}, \qquad\text{and}\qquad \zeta_j := \zeta_0-j\eta\quad(j\geq 1), \end{equation}
and note that with $\nu_0:=1$ and $\mu_0 :=\zeta_0$, we have
\[ u_0 = x^{1/(100\log_2{x})} < x^{1/\log_2{x}} \leq P^{+}(p_0-1), P^{+}(q_0+1) \leq x = v_0, \]
by condition (7) in the definitions of $\A_{\phi}$ and $\A_{\sigma}$.
%Moreover, by condition (6) in the definition of $\phi$, and the inequalities $x_2, y_2 \leq 0.9$, we have
%$P^{+}(p_i-1), P^{+}(q_j+1) \leq x^{1/10\log_2{x}}=v_0$ for any $i, j \geq 1$. We will use this below to show
%that we can secure condition (a) of Lemma \ref{lem51}.
To choose the remaining $\mu_j$ and $\nu_j$, it is helpful to know that $p_j$ and $q_j$ are close together
(renormalized on a double logarithmic scale) for $1 \leq j \leq k$. This is the substance of the following lemma.
\begin{lem}\label{lem:length1} If $p_j\geq S$ and $q_j \geq S$, then
\begin{equation}\label{eq:xjyjclose} \left|x_j - y_j\right| \leq (2j+1)\delta < \eta.
\end{equation}
These hypotheses hold if $L-j \geq 2C\log_4{x}+12$, and so in particular for $1 \leq j \leq k$.
\end{lem}
\begin{proof} Suppose for the sake of contradiction that $y_j \geq x_j + (2j+1)\delta$; since the $p_i$ and $q_i$ are all $S$-normal, this would imply that
 \[ (j+1)(y_j-x_j - \delta) \leq \frac{\Omega(\sigma(a'),p_j, q_j)}{\log_2{x}} = \frac{\Omega(\phi(a),p_j,q_j)}{\log_2{x}}\leq j(y_j-x_j + \delta), \]
which is false. We obtain a similar contradiction if we suppose that $x_j \geq y_j + (2j+1)\delta$. The second half of the lemma follows
from Lemma \ref{lem:xbounds} and a short calculation, together with the estimate $k\sim (1/2-\epsilon/10)L$ of Lemma \ref{lem:minimalindex}.%This gives \eqref{eq:xjyjclose}. Now we note that since $p_j$ and $q_j$ are $S$-normal, $x_j^{\ast}$ is close to $x_j$ and $y_j^{\ast}$ is close to $y_j$. Indeed,
%\[ p_i \geq P^{+}(p_i-1) \geq (p_i-1)^{1/3\log_2{x}} \geq p_i^{1/4\log_2{x}}, \]
%and similar calculation goes through for $P^{+}(q_i+1)$.  Applying $\log_2$ gives \eqref{eq:xjxjast}.
\end{proof}

We choose the intervals $[\mu_j, \nu_j]$ for $1 \leq j \leq k-1$ successively, starting with $j=1$. (We select $\nu_k$ last,
by a different method.) Say that the pair $\{x_j, x_{j+1}\}$ is \emph{well-separated} if $x_j - x_{j+1} \geq 10\eta$, and \emph{poorly separated} otherwise.

In the well-separated case,
among all $\zeta_i$ (with $i \geq 0$), choose $\zeta$ minimal and $\zeta'$ maximal with
\begin{align*} \zeta' \log_2{x} &\leq \log_2 \min\{P^{+}(p_j-1), P^{+}(q_j+1)\}\\
&\leq \log_2\max\{P^{+}(p_j-1), P^{+}(q_j+1)\} \leq \zeta \log_2{x}, \end{align*} and put
\[ \mu_j := \zeta, \qquad \nu_j := \zeta'. \]
In the poorly-separated case,  $j < k-1$, by \eqref{eq:finalseparation}.
We select $[\mu_j, \nu_j]=[\mu_{j+1},\nu_{j+1}]$ by a similar recipe:  Among all $\zeta_i$ (with $i \geq 0$), choose $\zeta$ minimal and $\zeta'$ maximal with
\begin{align*} \zeta' \log_2{x} &\leq \log_2 \min\{P^{+}(p_j-1), P^{+}(q_j+1), P^{+}(p_{j+1}-1), P^{+}(q_{j+1}+1)\}\\
&\leq \log_2\max\{P^{+}(p_j-1), P^{+}(q_j+1), P^{+}(p_{j+1}-1), P^{+}(q_{j+1}+1)\} \leq \zeta \log_2{x}, \end{align*}
and put
\[ \nu_j = \nu_{j+1} = \zeta, \quad\text{and}\quad \mu_j = \mu_{j+1} = \zeta'.\]

To see that these choices are well-defined, note that by (7) in the definition of $\A_f$, we have $x_j, y_j \leq \zeta_0$, which implies that a suitable choice of $\zeta$ above exists in both cases.
Also, for $1 \leq i \leq k$, we have $x_i, y_i \geq (\log_2{x})^{-1/2+\epsilon/11}$ (by Lemma \ref{lem:minimalindex} and \ref{lem:xbounds}(ii)). So by Lemma \ref{lem:length2},
\[ \log_2 \min\{P^{+}(p_i-1), P^{+}(q_i+1)\}/\log_2{x}  \geq (\log_2{x})^{-1/2+\epsilon/12}, \]
say. Since neighboring $\zeta_i$ are spaced at a distance $\eta \asymp (\log_2{x})^{-1/2}(\log_3{x})^{3/2}$, a suitable choice of $\zeta'$ also exists in both cases.

For our application of Lemma \ref{lem51}, it is expedient to keep track at each step of the length of the intervals $[\mu_j, \nu_j]$, as well as
the distance between the left-endpoint of the last interval chosen and the right-endpoint of the succeeding interval (if any).
In the well-separated case, Lemmas \ref{lem:length1} and \ref{lem:length2} show that
\[ \nu_j \leq \max\{x_j, y_j\} + \eta \leq x_j + 2\eta, \]
while
\begin{align} \mu_j &\geq \min\{x_j, y_j\} - \frac{\log_3{x}+\log{4}}{\log_2{x}} - \eta\notag\\ &\geq x_j-3\eta\label{eq:mujlower1},
\end{align}
so that
\[ \nu_j - \mu_j \leq 5\eta.\]
Also, if a succeeding interval exists (so that $j +1 \leq k-1$), then
\[ \nu_{j+1} \leq \max\{x_{j+1}, y_{j+1}\} + \eta \leq x_{j+1} + 2\eta, \]
and the separation between $\mu_j$ and $\nu_{j+1}$ satisfies the lower bound
\begin{equation}\label{eq:musepwell}
 \mu_j - \nu_{j+1} \geq x_j - x_{j+1} - 5\eta \geq 5\eta. \end{equation}
In the poorly separated case, we have
\[ \nu_j \leq \max\{x_j,y_j,x_{j+1},y_{j+1}\} + \eta = \max\{x_j, y_j\} + \eta \leq x_j + 2\eta, \]
as before,  but the lower bound on $\mu_j$ takes a slightly different form;
\begin{align} \notag\mu_j &\geq \min\{x_j, y_j, x_{j+1}, y_{j+1}\} - \frac{\log_3{x}+\log{4}}{\log_2{x}} - \eta \\
&\geq (x_{j+1}-\eta) - \frac{\log_3{x}+\log{4}}{\log_2{x}} - \eta \geq x_{j+1} - 3\eta \geq x_j - 13\eta,
\label{eq:mujlower2}\end{align}
so that
\[ \nu_j - \mu_j \leq 15\eta.\]
In this case, since $\nu_j = \nu_{j+1}$ and $\mu_j = \mu_{j+1}$,  the succeeding interval (if it exists) is
$[\mu_{j+2}, \nu_{j+2}]$. By Lemma \ref{lem:xbounds}(iii),
\[ x_j - x_{j+2} \geq 0.1 x_j \geq 0.1 (\log_2{x})^{-1/2+\epsilon/10} > 20 \eta, \]
say. Thus,
\[ \nu_{j+2} \leq \max\{x_{j+2}, y_{j+2}\} + \eta \leq x_{j+2}+2\eta \leq x_j - 18\eta,\]
and so
\begin{equation}\label{eq:museppoor} \mu_{j+1} - \nu_{j+2} = \mu_j - \nu_{j+2} \geq (x_j-13\eta)-(x_j-18\eta) \geq 5\eta. \end{equation}

At this point we have selected intervals $[\mu_j, \nu_j]$, for all $0 \leq j \leq k-1$. We
choose $\nu_k= \zeta$, where $\zeta$ is the minimal $\zeta_i$ satisfying $\zeta \geq x_k + \eta$. Note that
\[ \log_2{S}/\log_2{x} = 36 \log_3{x}/\log_2{x} < (\log_2{x})^{-1/2+\epsilon/11} \leq x_k < \zeta = \nu_k \leq x_k + 2\eta. \]
Thus, $v_k > S$. From \eqref{eq:mujlower1} and \eqref{eq:mujlower2},
\[ \mu_{k-1} \geq x_{k-1} -3\eta, \]
so that also
\begin{equation}\label{eq:lastsep}
 \mu_{k-1} - \nu_k \geq x_{k-1}-x_k -5\eta \geq 15\eta,\end{equation}
where the last estimate uses \eqref{eq:finalseparation}.

\subsection{Verification of hypotheses} We now check that Lemma \ref{lem51} may be applied with $y=x$. By construction,
$S \leq v_k \leq v_{k-1} \leq \dots \leq v_0 = x$, and $u_i \leq v_i$ for all $0 \leq i \leq k-1$. Moreover,
if $[\mu_j,\nu_j]\neq [\mu_{j-1}, \nu_{j-1}]$ (where $2 \leq j \leq k-1$), then from \eqref{eq:musepwell} and \eqref{eq:museppoor},
\[ \mu_{j-1} - \nu_j \geq 5\eta = 50L\delta > 2\delta, \]
and from \eqref{eq:lastsep},
\[ \mu_{k-1} - \nu_{k} \geq 15\eta > 2\delta. \]
Thus, condition (a) of Lemma \ref{lem51} is satisfied. It follows from our method of selecting the $\mu_j$ and
$\nu_j$ that if $\nu_j = \nu_{j+1}$, then (again by \eqref{eq:museppoor}) $\nu_{j+2}\leq \mu_{j+1} - 5\eta < \nu_{j+1}=\nu_j$, which shows that condition (b) is also satisfied. Moreover, since $\nu_k > x_k$, we have $P^{+}(d) \leq p_L \leq p_k < v_k$.
%, and
%\[ d \leq p_L p_{L+1} p_{L+2} \cdots \leq p_L^{10\log_2{x}}. \]
%Since
%\begin{align}\notag\log_2{p_L^{10\log_2{x}}} &\ll \log_3{x} + \log_2{p_L} \ll \log_3{x} + \rho^{L} \log_2{x} \\ &\ll \log_3{x} + \rho^{-2\sqrt{\log_3{x}}} \rho^{L_0} \log_2{x} \ll \rho^{-2\sqrt{\log_3{x}}} \log_3{x} \ll \exp(O(\sqrt{\log_3{x}})), \label{eq:logbound}\end{align}
%we find that $d \leq x^{o(1)}$, and so certainly $d \leq x^{1/100}$.
%
So we may focus our attention on hypotheses (i)--(v) of Lemma \ref{lem51}. We claim that these hypotheses are satisfied with
our choices of $d$, $e$, and $f$ from \S\ref{sec:smalltiny} and with
\begin{equation}\label{eq:ldef} l:=L-k. \end{equation}
Property (i) is contained in (2) from the definition of $\A_f$. By construction,
\[ u_i \leq P^{+}(p_i-1), P^{+}(q_i+1) \leq v_i \]
for all $0 \leq i \leq k-1$, which is (ii). Since $v_k \geq S > \log{y}$, property (iii) holds by (1) in the
definition of $\A_f$. The verification of (iv) is somewhat more intricate. Recalling that $\nu_k > x_k$, it is clear from \eqref{eq:fddefs} 
that
\[ P^{+}(f) < p_k \leq v_k. \]
To prove the same estimate for $P^{+}(e)$, we can assume $e\neq 1$. Let $r = P^{+}(e)$, and observe that
$r \mid \sigma(R)$, for some prime power $R$ with $R \parallel q_k q_{k+1} \cdots$. If $R$ is a
proper prime power, then from (1) in the definition of $\A_f$, we have $r \leq \sigma(R) \leq 2R \leq 2(\log{x})^2 < v_k$. So we can assume that $R$ is prime, and so $R \leq q_k$ and \[ r \leq P^{+}(R+1) \leq \max\{3, R\} \leq q_k.\] But by Lemma \ref{lem:length1},
\[ \log_2{q_k}/\log_2{x} = y_k \leq x_k + (2k+1)\delta < x_k + \eta \leq \nu_k. \]
Thus, $P^{+}(e) = r \leq v_k$. Hence, $P^{+}(ef) \leq v_k$. Turning to the second half of (iv),
write $p_k\cdots p_{L-1} = AB$, where $A$ is squarefree, $B$ is squarefull and $\gcd(A,B)=1$. Recalling \eqref{eq:normalpomega},
we see that
\[ \Omega(\phi(A)) \leq 3\Omega(A)\log_2{v_k} \leq 3l \log_2{v_k}, \]
with $l$ as in \eqref{eq:ldef}. Let $B'$ be the largest divisor of $a$ supported on the primes
dividing $B$, so that $B'$ is squarefull and $B\mid B'$. By (1) in the definition of $\A_f$, we have
$B' \leq (\log{x})^2$. If $B' \leq \exp((\log_2{x})^{1/2})$, then (estimating crudely) \[ \Omega(\phi(B)) \leq \Omega(\phi(B')) \leq 2\log{\phi(B')} \leq 2\log{B'} \leq 2(\log_2{x})^{1/2}.\]
On the other hand, if $B' > \exp((\log_2{x})^{1/2})$, then by (4) in the definition of $\A_f$, \[ \Omega(\phi(B)) \leq \Omega(\phi(B')) \leq 10\log_2{\phi(B')} \leq10\log_2{B'} \ll \log_3{x}.\]
Since $\log_2{v_k} = \nu_k \log_2{x} > \eta \log_2{x} > (\log_2{x})^{1/2}$, we have
we have $\Omega(\phi(B)) \leq 2\log_2{v_k}$ in either case. Hence, \[ \Omega(f) = \Omega(\phi(A)) + \Omega(\phi(B)) \leq (3l +2) \log_2{v_k} \leq 4l\log_2{v_k}, \]
which completes the proof of (iv). Finally, we prove (v): Suppose that $b \geq x^{1/3}$ is a divisor of
$p_0-1$. Recalling again \eqref{eq:normalpomega},
\[ P^{+}(b) \geq b^{1/\Omega(p_0-1)} \geq b^{\frac{1}{3\log_2{x}}} \geq x^{\frac{1}{9\log_2{x}}} > x^{\frac{1}{100\log_2{x}}}\geq v_1. \]
Thus, setting $b$ to be the largest divisor of $p_0-1$ supported on the primes $\leq v_1$, we have $b < x^{1/3}$.
From \eqref{eq:p0large} and conditions (0) and (7) in the definition of $\A_\phi$,
\[ p_0 = \frac{a}{p_1 p_2 p_3 \cdots} > \frac{x/\log{x}}{x^{1/100} p_1} > x^{0.95}, \]
say.  Thus, $(p_0-1)/b$ is a divisor of $p_0-1$ composed of primes $> v_1$ and of size at least $(p_0-1) x^{-1/3} > x^{9/10} x^{-1/3} > x^{1/2}$.
\subsection{Denouement} We are now in a position to establish Proposition \ref{prop:main} and so also Theorem \ref{thm:main}. Suppose that $k$ and the $\mu_i$ and $\nu_i$ are fixed, as is $d$; this also fixes $l=L-k$. By Lemma \ref{lem51}, whose hypotheses were verified above, the number of values $\phi(a)$ coming from corresponding solutions to $\phi(a)=\sigma(a')$, with $(a,a') \in \A_{\phi}\times\A_{\sigma}$, is
\begin{multline}
\ll \frac{x}{d} (c\log_2{x})^{6k} (k+1)^{\Omega(d)} (\log{v_k})^{8(k+l)\log{(k+1)} + 1} (\log{x})^{-2+\sum_{i=1}^{k-1}a_i\nu_i + E} \\
\leq \frac{x}{d} \exp(O((\log_3{x})^2)) L^{\Omega(d)} (\log{v_k})^{L^2}(\log{x})^{-2+\sum_{i=1}^{k-1}{a_i x_i}+E'},\label{eq:endgame0}
\end{multline}
where
\begin{align*} E' :&= E + \sum_{i=1}^{k-1}a_i(\nu_i - x_i)\\
&=\delta \sum_{i=2}^{k}{(i\log{i}+i)} + 2 \sum_{i=1}^{k-1} (\nu_i-\mu_i) +
\sum_{i=1}^{k-1}a_i(\nu_i - x_i). 
\end{align*}
By our choice of $\nu_i$ and $\mu_i$ in \S\ref{sec:ujvjchoice}, we have $\nu_i - \mu_i \ll \eta$ and $\nu_i -x_i \ll \eta$. Hence,
\[ E' \ll \delta L^2\log{L} + \eta\left(L+\sum_{i=1}^{k-1}a_i\right) \ll \delta L^2\log{L} + \eta L^2\log{L} \ll \delta L^3\log{L}.
\]
In combination with (7) from the definition of $\A_{\phi}$, this shows that the exponent of $\log{x}$ on the right-hand side of \eqref{eq:endgame0} is at most $-1-\omega+E' \leq -1-\omega/2$, and so
\[ (\log{x})^{-2+\sum_{i=1}^{k-1}{a_i x_i}+E'} \leq (\log{x})^{-1} \exp\left(-\frac{1}{2}(\log_2{x})^{1/2+\epsilon/2}\right).\]
Moreover, by Lemma \ref{lem:minimalindex} and Lemma \ref{lem:xbounds}(i),
\begin{equation}\label{eq:nukbound} \nu_k \leq x_k + 2\eta \leq (\log_2{x})^{-1/2+\epsilon/9} + 2\eta \leq (\log_2{x})^{-1/2+\epsilon/5},\end{equation}
and hence
\[ (\log{v_k})^{L^2} =\exp(L^2 (\log_2{x})\nu_k) \leq \exp((\log_2{x})^{1/2+\epsilon/4}).\]
Inserting all of this back into \eqref{eq:endgame0}, we obtain an upper bound which is
\begin{equation}\label{eq:beforesumming} \ll \frac{x}{\log{x}} \exp\left(-\frac{1}{3}(\log_2 x)^{1/2+\epsilon/2}\right) \frac{L^{\Omega(d)}}{d}. \end{equation}

Now we sum over the parameters previously held fixed. We have $k < L$; also, for $i > 0$, each $\mu_i$ and $\nu_i$ has the form $\zeta_j$ of \eqref{eq:zetajdef}. Thus, the the number of possibilities for $k$ and the $\mu_i$ and $\nu_i$ is
\begin{equation}\label{eq:munuchoices} \leq L (1+\lfloor \eta^{-1}\rfloor)^{2L} \leq \exp(O((\log_3{x})^2)).\end{equation}Next, we prove that
\begin{equation}\label{eq:uniformOmegad}
 \Omega(d) \ll (\log_2{x})^{1/2} \end{equation}
uniformly for the $d$ under consideration, so that
\begin{equation}\label{eq:uniformL} L^{\Omega(d)} \leq \exp(O((\log_2{x})^{1/2}\log_4{x})). \end{equation}
Put $m:=p_L p_{L+1} \cdots$. Suppose first that $p_L \neq p_{L-1}$, so that $m$ is a unitary divisor of $a$
and $d=\phi(m)$. If $m \leq \exp((\log_2{x})^{1/2})$, then \eqref{eq:uniformOmegad} follows from the crude bound $\Omega(d) \ll \log{d}$. On the other hand, if $m > \exp((\log_2{x})^{1/2})$, then from (4) in the definition of $\A_\phi$, we have $\Omega(d) = \Omega(\phi(m)) \ll \log_2{m}$. But by (3) in the definition of $\A_{\phi}$ and Lemma \ref{lem:xbounds}(i),
\begin{multline*} \log_2{m} \leq \log_2{p_L^{10\log_2{x}}} \ll \log_3{x} + \log_2{p_L} \ll \log_3{x} + \rho^{L} \log_2{x} \\ \ll \log_3{x} + \rho^{-2\sqrt{\log_3{x}}} \rho^{L_0} \log_2{x} \ll \rho^{-2\sqrt{\log_3{x}}} \log_3{x} \ll \exp(O(\sqrt{\log_3{x}})), \end{multline*}
which again gives \eqref{eq:uniformOmegad}. Suppose now that $p_L = p_{L-1}$. In this case, let $m'$ be the largest divisor of $a$ supported on the primes dividing  $m$. Then $d \mid \phi(m')$, and so $\Omega(d) \leq \Omega(\phi(m'))$. Write $m' = p_L^{j} m''$, where $j \geq 2$ and $p_L \nmid m''$; both $p_L^{j}$ and $m''$ are unitary divisors of $a$. We have $\Omega(\phi(m'')) \ll (\log_2{x})^{1/2}$,
by mimicking the argument used for $m$ in the case when $p_L \neq p_{L-1}$. 
Also, $\Omega(\phi(p_L^{j})) \ll (\log_2{x})^{1/2}$ except possibly if $p_L^{j} > \exp((\log_2{x})^{1/2})$, in which case, invoking (1) and (4) in the definition of $\A_{\phi}$,
\[ \Omega(\phi(p_L^{j})) \leq 10\log_2{\phi(p_L^{j})} \leq 10 \log_2{p_L^{j}} \leq 10\log_2{(\log^2{x})} \ll \log_3{x}. \]
So \[ \Omega(d) \leq \Omega(\phi(m')) = \Omega(\phi(p_L^j))+\Omega(\phi(m'')) \ll (\log_2{x})^{1/2},\] confirming \eqref{eq:uniformOmegad}.

Referring back to \eqref{eq:beforesumming}, we see that it remains to only to estimate the sum of $1/d$. Since $P^{+}(d) \leq v_k$, \eqref{eq:nukbound} shows that every prime dividing $d$ belongs to the set $\Pp:=\{p: \log_2 p \leq (\log_2{x})^{1/2+\epsilon/5}\}$. Thus,
\begin{equation}\label{eq:sum1d} \sum\frac{1}{d} \leq \prod_{p \in \Pp}\left(1+\frac{1}{p}+\frac{1}{p^2} + \dots\right) \ll \exp((\log_2{x})^{1/2+\epsilon/5}).\end{equation}

Combining the estimates \eqref{eq:beforesumming}, \eqref{eq:munuchoices}, \eqref{eq:uniformL}, and \eqref{eq:sum1d}, we find that
\[ \#\{\phi(a): a\in \A_{\phi}, a'\in \A_{\sigma}, \phi(a)=\sigma(a')\} \ll
\frac{x}{\log{x}} \exp\left(-\frac{1}{4}(\log_2 x)^{1/2+\epsilon/2}\right),\]
which completes the proof of Proposition \ref{prop:main} and of Theorem \ref{thm:main}.
%( $u_i =\exp((\log{x})^{\mu_i})$ and $v_i = u_i =\exp((\log{x})^{\nu_i})$.)

\bibliographystyle{amsalpha}
\bibliography{phisigma}
\end{document}